\documentclass{amsart}
\usepackage{enumerate}
\usepackage{mathtools}
\usepackage{amsfonts}
\usepackage{amssymb}
\usepackage{amsthm}
\usepackage{mathrsfs}
\usepackage{graphicx}
\usepackage{url}
\usepackage[margin=1in]{geometry}

\newtheorem{thm}{Theorem}[section]
\newtheorem{prop}[thm]{Proposition}

\newtheorem{cor}[thm]{Corollary}

\newtheorem{lemma}[thm]{Lemma}
\newtheorem*{thm*}{Theorem}						% Suppress numbering
\newtheorem*{prop*}{Proposition}
\newtheorem*{lemma*}{Lemma}
\newtheorem*{cor*}{Corollary}
\newtheorem*{conj*}{Conjecture}

\theoremstyle{definition}
\newtheorem{definition}[thm]{Definition}
\newtheorem{example}[thm]{Example}
\newtheorem{remark}[thm]{Remark}

\newcommand{\Z}{\mathbb{Z}}

\newcommand{\R}{\mathbb{R}}

\renewcommand{\P}{\mathbb{P}}
\DeclareMathOperator{\grass}{G}
\DeclareMathOperator{\pgrass}{\mathbb{G}}
\DeclareMathOperator{\gl}{GL}
\DeclareMathOperator{\eff}{Eff}
\DeclareMathOperator{\ceff}{\overline{\eff}}
\DeclareMathOperator{\pic}{Pic}
\begin{document}

\title[Effective cycles on blow-ups of $\grass(k,n)$]{Effective cycles on blow-ups of Grassmannians}

\author[J. Kopper]{John Kopper}
\address{Department of Mathematics, Statistics and CS \\ University of Illinois at Chicago, Chicago, IL 60607}
\email{jkoppe2@uic.edu}

\thanks{During the preparation of this article the author was partially supported by NSF RTG grant DMS-1246844.}

\subjclass[2010]{Primary: 14C25, 14C99, 14M15, 14N15. Secondary: 14E30, 14E99, 14M07}
\keywords{Cones of effective cycles, higher codimension cycles, blow-ups of Grassmannians}

\begin{abstract}
In this paper we study the pseudoeffective cones of blow-ups of Grassmannians at sets of points. For small numbers of points, the cones are often spanned by proper transforms of Schubert classes. In some special cases, we provide sharp bounds for when the Schubert classes fail to span and we describe the resulting geometry.
\end{abstract}

\maketitle

\section{Introduction}
Positive cones of divisors and curves play an important role in algebraic geometry and have been a subject of great interest for many years. Recently there has been increased interest in cones of higher codimension cycles (\cite{chen-coskun}, \cite{clo}, \cite{delv}, \cite{djv}, \cite{fulger-lehmann}), but only a small number of examples have been studied. The goal of this paper is to contribute to the body of known examples. This paper was inspired by the recent work of Coskun-Lesieutre-Ottem \cite{clo} who studied effective cones on blow-ups of projective space.

Let $X=\grass(k,n)$ be the Grassmannian of $k$-dimensional subspaces of a fixed $n$-dimensional vector space over an algebraically closed field. $X$ is a rational homogeneous variety of dimension $k(n-k)$. In this paper we will assume $k \geq 2$ and $n-k \geq 2$. Let $\Gamma$ be a set of $r$ distinct points of $X$, and $X_\Gamma$ the blow-up of $X$ along $\Gamma$. We will denote by $\ceff_m(X_\Gamma)$ the closed convex cone of pseudoeffective dimension-$m$ cycles of $X_\Gamma$, and by $\ceff^m(X_\Gamma)$ the cone of pseudoeffective codimension-$m$ cycles.

Computations of the effective cones in this paper follow roughly three distinct methods. First, if a connected solvable group $B$ acts on a variety, then the effective cones are determined by $B$-invariant effective cycles (Section \ref{sec:group_actions}). Second, we can produce nef classes that impose numerical conditions on effective cycles of complementary dimension (see Sections \ref{sec:divisors} and \ref{sec:examples}). Lastly, we can push cycles forward from a subvariety $Y\subset X_\Gamma$ when $\ceff_m(Y)$ is understood (Section \ref{sec:infinitely_gen}).

Even in the case of divisors, little is known about the structure of the effective cone when $\Gamma$ consists of more than a few points. It is not known, for example, how many points need to be blown-up for $X_\Gamma$ to cease to be a Mori Dream Space or log Fano. However, if $r$ is sufficiently small, then $\ceff_m(X_\Gamma)$ often admits a simple description in terms of pullbacks of the Schubert classes on $\grass(k,n)$.

\begin{definition}
Let $X_\Gamma$ denote the blow-up of the Grassmannian $\grass(k,n)$ at a finite set of points $\Gamma$. We say that $\ceff_m(X_\Gamma)$ is \emph{S-generated} if it is the cone spanned by the classes of $m$-dimensional linear spaces in the exceptional divisors and the classes of strict transforms of $m$-dimensional Schubert classes on $\grass(k,n)$, possibly passing through the points of $\Gamma$. If $\alpha$ is in the span of these classes, we will say that $\alpha$ is \emph{in the span of the Schubert classes.} We say that $\ceff_m(X_\Gamma)$ is \emph{finitely generated} if it is a rational polyhedral cone.
\end{definition}

Our first result says that when $\Gamma$ consists of a single point, the pseudoeffective cone in all dimensions is necessarily S-generated.

\begin{cor*}[\ref{cor:lin_gen1pt}]
Let $X_1$ be the blow-up of $\grass(k,n)$ at a point. Then $\ceff_m(X_1)$ is S-generated for $0 \leq m < k(n-k)$.
\end{cor*}

Using a description of the singularities of Schubert varieties we can use the above result to explicitly describe $\ceff_m(X_1)$. This recovers a result of Rischter \cite{rischter} regarding the effective cone of divisors on blow-ups of Grassmannians at a single point. When $n=2k$ and $\Gamma$ is a set of two general points, we are able to give a similarly explicit description of $\ceff^1(k,2k)$.

\begin{thm*}[\ref{thm:divisors_2pts}]
Let $\Gamma$ be a set of two general points on $\grass(k,2k)$ and $X_\Gamma$ the blow-up of $\grass(k,2k)$ along $\Gamma$. Then $\ceff^1(X_\Gamma)$ is S-generated. In particular it is spanned by the exceptional classes and classes of the form $H-mE_1 - (k-m)E_2$, where $0 \leq m \leq k$ and $H$ is the pullback of the hyperplane class on $\grass(k,2k)$ in its Pl\"ucker embedding.
\end{thm*}

Theorem \ref{thm:divisors_2pts} was also shown independently by Rischter in \cite{rischter}. We also give a description of $\ceff_m(X_\Gamma)$ when $\Gamma$ is two general points and $m$ is arbitrary.

\begin{thm*}[\ref{thm:2pts}]
Let $\Gamma$ be a set of 2 general points on $\grass(k,n)$ and $X_\Gamma$ the blow-up of $\grass(k,n)$ along $\Gamma$. Then $\ceff_m(X_\Gamma)$ is generated by strict transforms of Schubert cycles and strict transforms of intersections (not necessarily transverse) of Schubert cycles.
\end{thm*}

The proof of Theorem \ref{thm:2pts} will show that $\ceff_m(X_\Gamma)$ is in fact finitely generated. It is more difficult to describe $\ceff_m(X_\Gamma)$ when $\Gamma$ consists of many points. To constrain the problem, we restrict our attention to specific values of $m$. The cases $m=1$ and $m=2$ are particularly tractable because the intersection product on Grassmannians has a simple description in (co)dimensions 1 and 2. By studying the Pl\"ucker embedding of $\grass(k,n)$, one can often construct nef classes of divisors as restrictions of hypersurfaces in projective space. These classes impose numerical conditions on effective 1- and 2-cycles. Using this idea, Rischter proves in \cite{rischter} that if $\Gamma$ is a set of $r$ general points on $\grass(k,n)$, then the cone of curves $\ceff_1(X_\Gamma)$ is S-generated for $r$ at most the codimension of $\grass(k,n)$ in its Pl\"ucker embedding plus one. We show that the same holds for $\ceff_2(X_\Gamma)$.

\begin{prop*}[\ref{prop:lingen_terriblebound}]
Let $\Gamma$ be a set of $r$ general points on $\grass(k,n)$. Then the cones $\ceff_1(X_\Gamma)$ and $\ceff_2(X_\Gamma)$ are S-generated if
\[
r \leq \binom{n}{k} -k(n-k).
\]
\end{prop*}

In some circumstances this can be improved by one for curves (Proposition \ref{prop:lingen_prettybadbound}). We do not expect this to be sharp. Indeed we show that if the points of $\Gamma$ are very general then we can characterize precisely when $\ceff_1(X_\Gamma)$ is S-generated.

\begin{cor*}[\ref{prop:lingen_verygeneral}]
Let $X_r$ be the blow-up of $\grass(k,n)$ at $r$ very general points. Then $\ceff_1(X_r)$ is S-generated if and only if $r \leq \deg \grass(k,n)$ under the Pl\"ucker embedding $\grass(k,n) \to \P^{\binom{n}{k} -1}$.
\end{cor*}

It is not clear whether the assumption ``very general'' can be replaced with ``general.'' For example, in the case of $\grass(2,4)$ it can be. Indeed, if the points of $\Gamma$ are general and $r=2=\deg \grass(2,4)$ then $\ceff_1(X_\Gamma)$ is S-Generated (Proposition \ref{prop:curves_g24}). Coskun-Lesieutre-Ottem \cite{clo} use a similar argument to describe curves on blow-ups of $\P^n$ at very general points. By taking cones over lower-dimensional cycles, they are able to inductively describe pseudoeffective cones of higher dimension. Their argument does not immediately generalize to $\grass(k,n)$. One approach is to consider $\pgrass(k,n)$, the Grassmannian of $k$-planes in $\P^n$. If $Z \subset \pgrass(k,n)$ is a subvariety then
\[
Y = \bigcup_{W \in Z} W \subset \P^n
\]
is a subvariety of projective space. By embedding this in $\P^{n+1}$ and taking cones over the $k$-planes $W$, we get a variety $Y'$ in $\P^{n+1}$ swept out by $(k+1)$-planes. That is, we obtain a subvariety $Z' \subset \pgrass(k+1,n+1)$. However, the $W'$ in $Z'$ are necessarily in special position: they are contained in the Schubert variety of $(k+1)$-planes containing the cone point. We explore this idea at the end of Section \ref{sec:infinitely_gen} to show that the behavior of $\ceff_1(X_\Gamma)$ depends dramatically on the configuration of the points of $\Gamma$ (c.f. Corollary \ref{cor:infinitely_gen_sigma3} and Proposition \ref{prop:points_on_a_line}).

Another difficulty, and one less peculiar to the case of Grassmannians, is that the description of pseudoeffective cones on blow-ups is intimately related to interpolation problems. Already in the case of $\P^2$ blown-up at 10 points we have only a conjectural description of the effective cone. Assuming this SHGH conjecture (\cite{harbourne}, \cite{gimi}, \cite{hirschowitz}), we show that $\ceff_1(X^{2,5}_r)$ is not finitely generated, where $X^{2,5}_r$ is the blow-up of $\grass(2,5)$ at $r \geq 6$ very general points. In fact, we show that this is true in the more general setting of Fano manifolds of index $n-1$ (Definition \ref{def:fanomfld}).

\begin{thm*}[\ref{thm:fanomfld_infinite}] Assume the SHGH conjecture holds for 10 very general points in $\P^2$. Let $Z$ be a Fano manifold of index $n-1$ such that $-K_Z = (n-1)H$ for $H$ ample. If $d=H^n$ and $X$ is the blow-up of $Z$ at $d+1$ very general points, then $\ceff_1(X)$ is not finitely generated.
\end{thm*}

\subsection*{Structure of this paper} In Section \ref{sec:preliminaries} we recall some facts about Grassmannians and pseudoeffective cones. In Section \ref{sec:group_actions} we study certain group actions on Grassmannians and compute their orbits. We then extend the group action to the blow-up of the Grassmannian at the fixed points of the action and a standard argument allows us to compute the resulting effective cones. Section \ref{sec:divisors} treats the case of divisors and curves. In Section \ref{sec:examples} we demonstrate our techniques in the case of some examples. In Section \ref{sec:infinitely_gen} we prove Theorem \ref{thm:fanomfld_infinite} to show that pseudoeffective cones on blow-ups of Fano manifolds of index $n-1$ are not finitely generated, assuming SHGH.

\subsection*{Acknowledgments} I am very grateful to my advisor, Izzet Coskun, for all of his help and guidance. I would also like to thank Dave Anderson, Tabes Bridges, Mihai Fulger, John Lesieutre, Eric Riedl, Rick Rischter, and Tim Ryan for their helpful conversations and feedback. I am also grateful for the referee's valuable input.

\section{Preliminaries}\label{sec:preliminaries}
In this section we collect some facts and lemmas that will be useful throughout the paper.

\begin{definition} Let $V$ be an $n$-dimensional vector space and $F_\bullet$ a complete flag
\[
F_0 \subset F_1 \subset \cdots \subset F_n = V.
\]
Let $\lambda = (\lambda_1, \dots, \lambda_k)$ be a partition such that $n-k \geq \lambda_1 \geq \cdots \geq \lambda_k \geq 0$. The \emph{Schubert variety} $\Sigma_\lambda(F_\bullet)$ is defined as the set of $W \in \grass(k,n)$ such that $\dim W \cap F_{n-k+i-\lambda_i} \geq i$ for $1 \leq i \leq k$.
If $F_\bullet$ and $F'_\bullet$ are two different flags, then the Schubert varieties $\Sigma_\lambda(F_\bullet)$ and $\Sigma_\lambda(F'_\bullet)$ represent the same Chow class. We therefore denote by $\sigma_\lambda = [\Sigma_\lambda(F_\bullet)]$ the corresponding class in $A^{|\lambda|}(\grass(k,n))$, where $|\lambda| = \sum \lambda_i$. The Schubert classes generate the Chow ring $A^\ast(\grass(k,n))$. We will typically suppress the flag in our notation and we will adopt the convention of omitting trailing zeros when writing partitions. We refer to \cite{eisenbud-harris} for details and proofs.
\end{definition}

The complete linear system $|\sigma_1|$ defines the \emph{Pl\"ucker embedding} $\grass(k,n) \to \P(\Lambda^k V)$. We may define the Pl\"ucker embedding in coordinates as follows. Given a $k$-dimensional subspace $W$ of an $n$-dimensional vector space $V$, we pick a basis $w_1, \dots, w_k$ of $W$. Then the map
\[
W \mapsto w_1 \wedge \cdots \wedge w_k \in \Lambda^k V
\]
is well-defined up to a constant, and therefore gives a well-defined map to $\P(\Lambda^k V)$.

We will often implicitly think of $\grass(k,n) = \grass(k,V)$ as embedded in $\P(\Lambda^k V)$ via the Pl\"ucker embedding. It will occasionally be convenient to identify $\grass(k,n)$ with $\pgrass(k-1,n-1)$, the Grassmannian of $(k-1)$-dimensional linear subspaces of $\P^{n-1}$.

In general it is not the case that every variety of class $\sigma_\lambda$ is a Schubert variety. In fact, varieties of class $\sigma_\lambda$ can sometimes pass through more general points than a Schubert variety $\Sigma_\lambda$. In some cases, however, all varieties $Y \subset \grass(k,n)$ of class $\sigma_\lambda$ are Schubert varieties. Such classes and their corresponding partitions are called \emph{rigid}. A complete classification of rigid Schubert varieties is given in \cite{coskun-rigidity}.

\subsection*{Singularities of Schubert varieties} Schubert varieties are not smooth in general. Let $\Sigma_\lambda$ be a Schubert variety corresponding to the partition $\lambda = (\lambda_1, \dots, \lambda_k)$. Then $\Sigma_\lambda$ is the disjoint union of the \emph{Schubert cells}
\[
\Sigma^\circ_\mu = \Sigma_\mu \setminus \left(\bigcup_{\nu > \mu}\Sigma_{\nu}\right)
\]
where $\mu=(\mu_1,\dots,\mu_k) \geq \lambda$. A formula for the multiplicity of $\Sigma_\lambda$ along a Schubert cell was given by Rosenthal-Zelevinsky.

\begin{thm}[\protect{\cite[Theorem 1]{rosenthal-zelevinsky}}]\label{thm:r-z} Let $\lambda$ and $\mu$ be partitions such that $\mu \geq \lambda$. Put $t_i = n-k+i - \lambda_i$. then the multiplicity of $\Sigma_\lambda$ along $\Sigma_\mu^\circ$ is
\[
(-1)^{s_1+\cdots+s_k} \det \begin{pmatrix}
\binom{t_1}{-s_1} & \cdots & \binom{t_k}{-s_k}\\
\binom{t_1}{1-s_1} & \cdots &  \binom{t_k}{1-s_k}\\
\vdots & & \vdots\\
\binom{t_1}{k-1-s_1} & \cdots & \binom{t_k}{k-1-s_k}\\
\end{pmatrix},
\]
where $s_i = \#\{\mu_j : \mu_j-j < \lambda_i-i\}$.
\end{thm}

By \cite[Lemma 4.1]{coskun-rigidity}, the maximum multiplicity of any variety of class $\sigma_\lambda$ is at most that of the Schubert variety $\Sigma_\lambda$.

\subsection*{Intersection theory on point blow-ups of Grassmannians} Numerical and rational equivalence coincide on point blow-ups of Grassmannians. Let $\Gamma$ be a set of $r$ points $p_1,\dots,p_r$ on $\grass(k,n)$, and let $\pi:X_\Gamma\to\grass(k,n)$ denote the blow-up of $\grass(k,n)$ along $\Gamma$. We will denote by $H$ pullback of the Schubert class $\sigma_1$ and by $E_i$ the exceptional divisor corresponding to $p_i$. The following intersection formulas are standard:
\[
	E_i^{k(n-k)}= (-1)^{k(n-k)+1}, \qquad H \cdot E_i = 0, \qquad E_i \cdot E_j = 0,\quad i\neq j.
\]
The exceptional divisors $E_i$ are isomorphic to $\P^{k(n-k)-1}$. We will denote by $E_i^{[m]}$ a codimension-$m$ linear cycle contained in $E_i$, and by $E_{i,[m]}$ a dimension-$m$ linear cycle contained in $E_i$. The Chow group $A_1(\grass(k,n))$ is free of rank one, generated by the Schubert cycle $\sigma_{n-k,n-k,\dots,n-k,n-k-1}$. We will denote by $\ell$ the pullback of this class via $\pi$, and by $\ell_i$ the class of a line contained in $E_i$. For the remaining Schubert classes we will abuse notation and use the symbol $\sigma_\lambda$ to denote their pullbacks in $X_\Gamma$.

The intersection product on Grassmannians is the subject of Schubert calculus, which we shall briefly describe. We refer to \cite{eisenbud-harris} for details and proofs. The two main results are \emph{Pieri's formula}, which describes the intersection product for \emph{special partitions}, i.e. partitions of the form $\lambda = (\lambda_1, 0 ,\dots, 0)$, and \emph{Giambelli's formula} which gives a method of writing any Schubert class as a polynomial in the special classes.

\begin{thm}[Pieri's formula \protect{\cite[Prop. 4.6]{eisenbud-harris}}]
Let $\lambda$ be a special partition and $\sigma_\mu$ any Schubert class, with $\mu=(\mu_1, \dots, \mu_k)$. Then
\[
	\sigma_\lambda \cdot \sigma_\mu = \sum_{\substack{\mu_i \leq \nu_i \leq \mu_{i-1}\\|\nu|=|\lambda|+|\mu|}} \sigma_\nu.
\]
\end{thm}

\begin{thm}[Giambelli's formula \protect{\cite[Prop. 4.16]{eisenbud-harris}}]
Let $\lambda = (\lambda_1, \dots, \lambda_k)$ be a partition and $\sigma_\lambda$ the corresponding Schubert class. Then
\[
	\sigma_\lambda = \det \begin{pmatrix}
	\sigma_{\lambda_1} && \sigma_{\lambda_{1}+1} && \sigma_{\lambda_1 + 2} && \cdots && \sigma_{\lambda_1 + k-1}\\
	\sigma_{\lambda_2 - 1} && \sigma_{\lambda_2} && \sigma_{\lambda_2 + 1} && \cdots && \sigma_{\lambda_2 + k-2}\\
	\vdots && \vdots && \vdots && \ddots && \vdots\\
	\sigma_{\lambda_k - k + 1} && \sigma_{\lambda_k - k + 2} && \sigma_{\lambda_k - k +3} && \cdots && \sigma_{\lambda_k}
	\end{pmatrix}.
\]
\end{thm}

\subsection*{Numerical spaces} Let $X = \grass(k,n)$ and let $\Gamma$ be a set of $r$ points on $X$. We denote by $X_\Gamma$ the blow-up of $X$ along $\Gamma$. For $0 \leq m \leq k(n-k)$, we write $N_m(X_\Gamma)$ for the $\R$-vector space of dimension-$m$ cycles on $X_\Gamma$ modulo numerical equivalence. We write $N^m(X_\Gamma)$ for the $\R$-vector space of codimension-$m$ cycles modulo numerical equivalence. The pullbacks of the Schubert classes $\sigma_\lambda$ such that $|\lambda| = m$, together with the exceptional classes $E_i^{[m]}$, $1 \leq i \leq r$, form an effective basis for $N^m(X_\Gamma)$.

\begin{definition}
A class in $N_m(X_\Gamma)$ is \emph{pseudoeffective} if it is the limit of classes of effective cycles. We denote by $\ceff_m(X_\Gamma)$ the closed convex cone in $N_m(X_\Gamma)$ consisting of pseudoeffective classes.
\end{definition}

The following standard lemma will be useful for describing effective cones.

\begin{lemma}[c.f. \cite{clo}]\label{lemma:positive_coeffs}
Let $Y \subset X_\Gamma$ be a subvariety of codimension $m$.
\begin{enumerate}
	\item If $Y \subset E_i$ for some $1 \leq i \leq r$ then $[Y] = bE_i^{[m]}$ for some $b > 0$.
	\item Otherwise, $Y = \sum_{|\lambda| = m} a_\lambda \sigma_\lambda - \sum_{i=1}^r b_iE_i^{[m]}$ with $a,b_i \geq 0$. The coefficient $b_i$ is equal to the multiplicity of $\pi(Y)$ in $\grass(k,n)$ at the point $p_i$.
\end{enumerate}
\end{lemma}

\section{Group actions on $\grass(k,n)$}\label{sec:group_actions} There is a natural action of the general linear group $\gl_n$ on $\grass(k,n)$. Fix an ordered basis $f_1,\dots,f_n$ for a vector space $V$ and let $B \subset \gl(V)$ be the group of upper-triangular matrices. It is well-known that $B$ is connected and solvable. The following lemma (see, e.g., \cite[Lemma 2.1]{anderson}) shows that this allows us to compute the pseudoeffective cones (c.f. \cite{fmss}).

\begin{lemma}\label{lemma:anderson}
Let $X$ be a complete irreducible variety of dimension $n$, and $B$ a connected solvable group acting on $X$. Then for any $0 \leq m \leq n$, the cone of effective classes in $N_m(X)$ is generated by the classes of $B$-invariant $m$-cycles.
\end{lemma}

Let $F_\bullet$ be the complete flag $F_0 \subset F_1 \subset \cdots \subset F_n = V$ where $F_i = \langle f_1, \dots, f_i \rangle$. By \cite[\S 9.4]{fulton}, the $B$-invariant $m$-orbits in $\grass(k,n)$ are precisely the Schubert cells $\Sigma^\circ_\lambda(F_\bullet)$ with $|\lambda| = k(n-k)-m$. This recovers the fact that the classes of the Schubert varieties generate $\ceff_m(\grass(k,n))$. Further, the point $F_k = \langle f_1, \dots, f_k \rangle$ is fixed by $B$, thus we can extend the $B$ action to the blow-up of $\grass(k,n)$ at $F_k$.

\begin{cor}\label{cor:lin_gen1pt}
Let $X_1$ be the blow-up of $\grass(k,n)$ at a point. Then $\ceff_m(X_1)$ is S-generated for $0 \leq m \leq k(n-k)$.

\end{cor}
\begin{proof}
Let $p \in \grass(k,n)$ be any point. Since $\grass(k,n)$ is homogeneous, we may choose a basis for $V$ so that the group $B$ of upper-triangular matrices fixes $p$. If $\pi:X_1 \to\grass(k,n)$ is the blow-up of $\grass(k,n)$ at $p$, then $B E = E$, where $E$ is the exceptional divisor. Let $Y\subset X_1$ be a $B$-invariant $m$-cycle. Since $\pi$ is $B$-equivariant, $\pi(Y)$ is $B$-invariant too. It follows that $\pi(Y)$ is either the point $p$ or a Schubert cell and $Y$ its strict transform. Thus $\ceff_m(X_1)$ is S-generated.
\end{proof}

Using Theorem \ref{thm:r-z} we can now explicitly describe the effective cones of $X_1$: the cone $\ceff^m(X_1)$ is spanned by $E^{[m]}$ and classes of the form $\sigma_\lambda - d_\lambda E^{[m]}$ where $|\lambda| = m$ and $d_\lambda$ is the multiplicity of the partition $\lambda$ along $\mu = (n-k,n-k, \dots, n-k)$.

A similar method works for the blow-up of $\grass(k,n)$ at two general points. The idea is to replace the group $B$ of upper-triangular matrices with a suitable alternative to ensure that both points are fixed by the action. First we will study what happens when $n=2k$, then reduce the general case to that one.

Let $V$ be a $2k$-dimensional vector space. We begin with two flags $F_\bullet$ and $G_\bullet$ of length $k$, defined by $F_i= \langle f_1, \dots, f_i\rangle$ and $G_j = \langle g_1, \dots, g_j \rangle$. We assume the set $\{f_1, \dots, f_k, g_1, \dots, g_k\}$ is a basis for $V$ so that $F_k$ and $G_k$ are in general position. We then have an action on $\grass(k,2k)$ by the group
\[
B=\left\{\begin{pmatrix}
B_1 & 0\\
0 & B_2
\end{pmatrix}\right\} \subset \gl(V),
\]
where $B_1$ and $B_2$ are the groups of $k\times k$ upper-triangular matrices. By construction, $B$ fixes $F_k$ and $G_k$. Our main goal is to compute the $B$-orbits of $\grass(k,2k)$ so that we can apply Lemma \ref{lemma:anderson}.

\begin{definition}
Let $W \subset V$ be any linear subspace. Define the $(k+1)\times(k+1)$ \emph{incidence matrix} $I_W$ by
\[
	I_W = \left( \dim W \cap (F_i+G_j)\right)_{0 \leq i,j \leq k}.
\]
That is, the $(i,j)$-th entry of $I_W$ is the dimension of the intersection $W \cap (F_i + G_j)$, where $0 \leq i,j \leq k$ (note the unusual indexing).
\end{definition}

The set of $W'$ satisfying $I_{W'} = I_W$ is clearly an intersection (not necessarily transverse) of Schubert cycles defined with respect to flags $H_\bullet$ of the form $H_\ell = F_{i_\ell} + G_{j_\ell}$. The main fact is that these are precisely the $B$-orbits of $\grass(k,2k)$. To prove this, we will use the following lemma which produces a computationally useful space $W'$ satisfying $I_{W'}=I_W$.

\begin{lemma}\label{lemma:incidence_matrices}
If $W \in \grass(k,2k)$ then there is some $W'$ with basis $w_1, \dots, w_k$ such that:
\begin{enumerate}
	\item Each $f_i$ and $g_j$ appears at most once in the expression of the $w_\ell$,
	\item $I_{W'} = I_W$.
\end{enumerate}
\end{lemma}
\begin{proof}
We construct $W' \in \grass(k,2k)$ as follows. Order the entries of $I_W$ with the lexicographic order on the indices, and let $w_1 = f_i + g_j$ where $(i,j)$ is the first nonzero entry. Inductively, let $w_\ell = f_i + g_j$ where $(i,j)$ is the first nonzero entry of $I_W-I_{\langle w_1, \dots, w_{\ell-1} \rangle}$.

Suppose for a contradiction that $w_\ell = f_i + g_j$ and $w_m = f_i + g_{j'}$, with $m \neq \ell$. Without loss of generality, assume $\ell < m$. Let $a$ be the $(i,j')$ entry of $I=I_W-I_{\langle w_1, \dots, w_{\ell-1}\rangle}$, $b$ the $(i,j'-1)$ entry and $c$ the $(i-1,j')$ entry, so that $I$ has the following form:
\[
I=\begin{pmatrix}
0 & 0 & \cdots & 0 & \ast & \cdots & \ast & \ast & \ast & \cdots\\
\vdots & \vdots &  & \vdots& \vdots& & \vdots& \vdots& \vdots& \\
0 & 0 & \cdots & 0 & \ast & \cdots & \ast & c & \ast & \cdots\\
0 & 0 & \cdots & 1 & \ast & \cdots & b & a & \ast & \cdots \\
0 & 0 & \cdots & \ast & \ast & \cdots & \ast & \ast & \ast & \cdots\\
\vdots & \vdots &  & \vdots& \vdots& & \vdots& \vdots& \vdots&
\end{pmatrix}
\]
We must have $b,c\leq a$. Further, $I_{\langle f_i + g_j\rangle}$ is the matrix whose nonzero entries are precisely the rectangle of 1's with corners $(i,j)$ and $(k,k)$. Since there exists $w \in W$ with $I_{\langle w \rangle} = I_{\langle f_i + g_j\rangle}$, $I$ is an incidence matrix corresponding to some subspace of $W$. Thus $a-1 \leq b,c$.

Now if $v \in V$ is any vector such that the $(i,j'-1)$ entry of $I-I_{\langle v \rangle}$ is strictly less than $b$, then the $(i,j')$ entry of $I-I_{\langle v \rangle}$ is strictly less than $a$. The same must hold for $c$ and $a$. By assumption, there is some $W' \subset V$ such that $(i,j')$ is the first nonzero entry of $I-I_{W'}$. It follows that $b=c=a-1$, and furthermore that every $w_r$ with $r < m$ must correspond to an entry in $I_W$ strictly above and to the left of $(i,j')$. This contradicts $w_\ell = f_i + g_j$. An analogous argument shows that the $g_j$ cannot occur more than once. This gives (1).

To prove (2), it suffices to show that $I_W-I_{W'} = 0$. We first show by induction that $I_W - I_{\langle w_1, \dots, w_\ell\rangle}$ has nonnegative entries for  $1\leq \ell \leq k$. Indeed, $I_W - I_{\langle w_1, \dots, w_\ell\rangle} = I_W - I_{\langle w_1, \dots, w_{\ell-1}\rangle} - I_{\langle w_\ell \rangle}$. Since $I_W-I_{\langle w_1, \dots, w_{\ell-1}\rangle}$ is the incidence matrix corresponding to a codimension-$(\ell-1)$ subspace of $W$, its rows are nondecreasing from left to right and its columns are nondecreasing from top to bottom. The incidence matrix $I_{\langle w_\ell \rangle}$ consists of all zeros except for a rectangle of 1's starting at the first nonzero entry of $I_W-I_{\langle w_1, \dots, w_{\ell-1}\rangle}$. It follows that $I_W-I_{\langle w_1, \dots, w_{\ell-1}, w_\ell\rangle}$ has nonnegative entries. Furthermore, for each $\ell$, the elements of $I_W-I_{\langle w_1, \dots, w_{\ell}\rangle}$ are bounded above by the $(k,k)$ entry, which is necessarily $k-\ell$. Since the $(k,k)$ entry of $I_W-I_{W'}$ is zero, claim (2) follows.
\end{proof}

\begin{lemma}\label{lemma:g(k,2k)_orbits}
Let $B$ be the set of block-diagonal invertible upper-triangular matrices. If $W \in \grass(k,2k)$ then the $B$-orbit of $W$ is the set of all $W' \in \grass(k,2k)$ such that $I_{W'} = I_W$.
\end{lemma}
\begin{proof}
Let $W'$ be as in Lemma \ref{lemma:incidence_matrices}. The general element of $B$ is of the form
\[
C=\begin{pmatrix}
x_1 & x_2 & x_3 & \cdots & x_k					&&&&&\\
0 & x_{k+1} & x_{k+2} & \cdots & x_{2k-1}		&&&&&\\
0 & 0 & x_{2k} & \cdots  &x_{3k-3}				&&&\makebox(0,0){\text{\huge0}}&&\\
\vdots & \vdots & \vdots & \ddots & \vdots		&&&&&\\
0 & 0 & 0 & \cdots & x_N						&&&&&\\
& & & &											&y_1 & y_2 & y_3 & \cdots & y_k\\
& & & &											& 0 & y_{k+1} & y_{k+2} & \cdots  &y_{2k-1}\\
& & \makebox(0,0){\text{\huge0}}& &											&0 & 0 & y_{2k} & \cdots  &y_{3k-3}	\\
& & & &											&\vdots & \vdots & \vdots & \ddots & \vdots	\\
& & & &											&0 & 0 & 0 & \cdots & y_N
\end{pmatrix},
\]
where $N = \frac{k(k+1)}{2}$ and $\prod x_iy_i \neq 0$. To simplify notation, let $x_{i,j}$ denote the variable in the $i$th row and $j$th column of $C$ for $1 \leq i,j \leq k$, and let $y_{i,j}$ be the $(k+i, k+j)$ entry. We show that for any $T \in \grass(k,2k)$ with $I_T = I_W$, we can choose the $x$ and $y$ variables so that $CW' = T$. Let $w_1, \dots, w_k$ be the basis for $W'$ constructed in Lemma \ref{lemma:incidence_matrices} and write $w_\ell = f_{i_\ell} + g_{j_\ell}$. Since $T$ has the same incidence dimensions as $W'$, $T$ necessarily admits a basis of the form
\[
t_\ell = \sum_{p=1}^{i_\ell} a_p f_p + \sum_{q=1}^{j_\ell} b_q g_q,
\]
with $a_{i_\ell}$ and $b_{j_\ell}$ nonzero. We can therefore set $x_{p,i_\ell}=a_p$ and $y_{q,j_\ell} = b_q$ so that $CW' = T$.
\end{proof}

\begin{cor}\label{cor:incidence_orbits_lowdimension}
If $W \subset V$ is any subspace of dimension $k' \leq k$, then the $B$-orbit of $W$ is the set of all $W' \in \grass(k',2k)$ such that $I_W = I_{W'}$. 
\end{cor}

We have thus computed the $B$-orbits of $\grass(k,2k)$. We now extend to the general case $n \geq 2k$ by modifying $B$ so that the additional columns have many degrees of freedom.

\begin{lemma}\label{lemma:2k+s}
Let $\grass(k,n)=\grass(k,2k+s)$ with $s>0$. Let $B$ be the set of invertible matrices of the form
\[
\sbox0{$\begin{matrix}
\ast&\ast &\cdots&\ast\\
0 & \ast & \cdots & \ast\\
\vdots & \vdots & \ddots & \vdots\\
0 & 0 & \cdots & \ast
\end{matrix}$}
\sbox1{$\begin{matrix}
\ast&\ast &\cdots&\ast\\
\vdots & \vdots & \vdots & \vdots\\
\ast & \ast & \cdots & \ast
\end{matrix}$}
C=\left(
\begin{array}{ccc}
\vphantom{\usebox{0}}\makebox[\wd0]{\huge $B_1$} & \makebox[\wd0]{\huge $0$} &\usebox{1}\\
\makebox[\wd0]{\huge $0$} & \makebox[\wd0]{\huge $B_2$} & \usebox{1}\\
\makebox[\wd0]{\huge $0$} & \makebox[\wd0]{\huge $0$} & \usebox{0}

\end{array}
\right)
\]
where $B_1$, $B_2$ are $k\times k$ invertible upper-triangular matrices. Let $F_\bullet $ and $G_\bullet$ be the length $k$ flags as above, corresponding to $B_1$ and $B_2$. If $W\in \grass(k,n)$ then the $B$-orbit of $W$ is the set of all $W' \in \grass(k,n)$ such that $I_W = I_{W'}$.
\end{lemma}
\begin{proof}
Let $W \in \grass(k,n)$, and let $V = W \cap (F_k + G_k)$. Then $V \in \grass(k',F_k+G_k) = \grass(k',2k)$, where $k' = \dim W \cap (F_k + G_k)$. By Corollary \ref{cor:incidence_orbits_lowdimension}, the $B$-orbit of $V$ is the set of all $V' \in \grass(k',2k)$ with $I_{V'} = I_{V}$. Let $w_{k'+1}, \dots, w_{k} \in W$ be linearly independent vectors so that $V + \langle w_{k'+1}, \dots, w_k \rangle = W$. Then the $B$-orbit of $\langle w_{k'+1}, \dots, w_k \rangle$ is the set of all $(k-k')$-dimensional spaces that do not meet $F_k + G_k$. This completes the proof.
\end{proof}

\begin{thm}\label{thm:2pts}
Let $\Gamma$ be a set of 2 general points on $\grass(k,n)$ and $X_\Gamma$ the blow-up of $\grass(k,n)$ along $\Gamma$. Then for all $0 \leq m \leq k(n-k)$, $\ceff_m(X_\Gamma)$ is generated by strict transforms of Schubert cycles and strict transforms of intersections of Schubert cycles defined with respect to flags defined in terms of the $f_i$ and $g_j$. In particular, $\ceff_m(X_\Gamma)$ is finitely generated for all $m$.
\end{thm}
\begin{proof}
For any Grassmannian $\grass(k,n)$, there is an isomorphism $\grass(k,n) \cong \grass(k,n-k)$. Without loss of generality, we may therefore assume $n \geq 2k$. Let $V$ be an $n$-dimensional vector space and choose a basis for $V$ so that the subgroup $B$ of upper-triangular matrices of the form in Lemma \ref{lemma:2k+s} fixes $\Gamma$. If $\pi:X_\Gamma \to\grass(k,n)$ is the blow-up of $\grass(k,n)$ at $\Gamma$, then the $B$-action extends to $X_\Gamma$ and $B E_i = E_i$, where $E_i$ are the exceptional divisors. Let $Y\subset X_\Gamma$ be a $B$-invariant irreducible $m$-cycle. Since $\pi$ is $B$-equivariant, $\pi(Y)$ is $B$-invariant too.

Since $\pi(Y)$ is $B$-invariant, it is the union of its $B$-orbits. By Lemma \ref{lemma:2k+s}, $\pi(Y)$ is an intersection of Schubert varieties. Thus $Y$ is the strict transform of an intersection of Schubert varieties. It follows that all extremal $m$-cycles are strict transforms of such intersections.
\end{proof}

It is reasonable to wonder if this argument extends to $\grass(k,3k)$ or indeed to $\grass(k,dk)$ for $d \geq 3$. Unfortunately, this does not appear to be the case. The dimension of $\grass(k,dk)$ is $(d-1)k^2$, but if $B \subset \gl_{3k}$ is the group of block-diagonal matrices with $d$ blocks of $k\times k$ upper-triangular matrices, then $B$ has dimension $dk(k+1)/2$. Thus for $d >2$ and large $k$, $B$ cannot have a dense orbit and therefore cannot have finitely many orbits.

\section{S-generation of cones of divisors and curves}\label{sec:divisors}
In this section we describe the pseudoeffective cones of divisors and curves on blow-ups of Grassmannians. We will denote $\grass(k,n)$ by $X$, and if $\Gamma$ is a set of points on $X$ then $X_\Gamma$ will denote the blow-up of $X$ along $\Gamma$. Recall that $H$ is the the pullback of the Schubert class $\sigma_1$ to $X_\Gamma$ and $\ell$ its dual, i.e., the pullback of the one-dimensional Schubert class. By Schubert calculus, we have $\ell \cdot H = 1$. Taking $\lambda = (1,0,\dots,0)$ in Corollary \ref{cor:lin_gen1pt}, we see that $\ceff^1(X_1)$ is generated by the exceptional class $E$ and $H - kE$.

A more geometric argument allows us to describe the pseudoeffective cone of divisors on the blow-up of $\grass(k,2k)$ at two points. The key ingredient is the explicit description of the singular locus of the Schubert variety $\Sigma_1$ given in Rosenthal-Zelevinsky \cite{rosenthal-zelevinsky}.

\begin{lemma}\label{lemma:dual_cone_2pts}
Let $\alpha = ae_0 - b_1e_1 - b_2e_2 \in \Z^3$ satisfy $ka \geq b_1+b_2$, with $a,b_1,b_2,k \geq 0$. Then $\alpha$ is in the positive linear span of the vectors $e_i$ for $0 \leq i \leq 2$ and $\beta_m = e_0 - me_1 - (k-m)e_2$ for $0 \leq m \leq k$.
\end{lemma}
\begin{proof}
We induct on $a$. Suppose $a=1$. By subtracting $e_0$ from $\alpha$, we may assume $ka = b_1 + b_2$. Then $\alpha = \beta_m$ for some $m$. For $a>1$, observe that if $\alpha$ satisfies the inequality $ka \geq b_1 + b_2$ then so does $\alpha - \beta_m$ for any $0 \leq m \leq k$.
\end{proof}

\begin{thm}\label{thm:divisors_2pts}
Let $\Gamma$ be a set of $r =2$ general points on $\grass(k,2k)$. Then $\ceff^1(X_\Gamma)$ is S-generated. Moreover, $\ceff^1(X_\Gamma)$ is spanned by the exceptional classes $E_1$, $E_2$, and classes of the form $H-mE_1 - (k-m)E_2$.
\end{thm}
\begin{proof}
Let $\alpha = aH - b_1E_1 - b_2E_2 \in \ceff^1(X_\Gamma)$. By Lemma \ref{lemma:dual_cone_2pts}, it suffices to establish (1) the inequality $ka \geq b_1 + b_2$ and (2) that $H - mE_1 - (k-m)E_2$ is effective for $0 \leq m \leq k$. To show the inequality we construct a moving curve of class $k\ell - \ell_1 - \ell_2$. Let $\Lambda_1$, $\Lambda_2$, $\Lambda_3$ be general $(k-1)$-planes in $\P^{2k-1}$. By \cite[Lemma 2.5]{coskun-lwr}, there is a degree $k$ scroll containing $\Lambda_1$, $\Lambda_2$ and $\Lambda_3$. This scroll corresponds to a degree $k$ curve in $\grass(k,2k)$ containing the points corresponding to the $\Lambda_i$. Thus in the blow-up $X_\Gamma$ of $\grass(k,2k)$ at $\Lambda_1$ and $\Lambda_2$, $k\ell - \ell_1 - \ell_2$ is a moving curve and has positive intersection with any effective divisor class. This gives the desired inequality.

Given two general $(k-1)$-planes $\Lambda_1$ and $\Lambda_2$, choose $m$ points on $\Lambda_1$ and $k-m$ points on $\Lambda_2$. Let $\Phi$ be the $(k-1)$-plane spanned by these points. Let $\Sigma_1$ be the Schubert variety of $(k-1)$-planes meeting $\Phi$. The singular locus of $\Sigma_1$ of multiplicity at least $m$ is the Schubert variety $\Sigma_{\lambda}$, where
\[
	\lambda_i = \begin{cases}
	m & 1 \leq i \leq m\\
	0 & m+1 \leq i \leq k.
	\end{cases}
\]
This is the Schubert variety of $(k-1)$-planes meeting $\Phi$ in dimension at least $m-1$. By construction, $\Lambda_1$ is in the multiplicity-$m$ locus of $\Sigma_1$ and $\Lambda_2$ is in the multiplicity-$(k-m)$ locus. Thus in the blow-up $X_\Gamma$ of $\grass(k,2k)$ at $\Gamma = \{\Lambda_1, \Lambda_2\}$, the class $H - m E_1 - (k-m)E_2$ is effective. This holds for any $m$ between 0 and $k$, so the proof is completed.
\end{proof}

We now consider curves. S-generation of $\ceff_1(X_\Gamma)$ is comparatively easy to determine because the effective generator of $N_1(\grass(k,n))$ can pass through precisely one general point.

\begin{prop}\label{prop:lines_through_one_pt}
Suppose $k \geq 2$ and $n-k \geq 2$ and $\sigma = \sigma_{n-k,n-k,\dots,n-k,n-k-1}$ is the one-dimensional Schubert class corresponding to $k$-planes contained in a fixed $(k+1)$-plane $\Lambda$ that contain a fixed $(k-1)$-plane $\Phi \subset \Lambda$. Then there is no variety of class $\sigma$ through more than one general point of $\grass(k,n)$.
\end{prop}
\begin{proof}
A curve on the Pl\"ucker embedding $\grass(k,n)\to \P^{\binom{n}{k}-1}$ is a line if and only if it has class $\sigma$. It follows that the class $\sigma$ is rigid, and therefore it suffices to prove the claim for Schubert varieties. Let $W$ and $W'$ be two general $k$-planes in an $n$-dimensional vector space $V$. Then $\dim (W+W') \geq 2+k$. But if $W$ and $W'$ are contained in $\sigma$ then they must be contained in the same $(k+1)$-plane, a contradiction.
\end{proof}

It follows that $\ceff_1(X_\Gamma)$ is S-generated if and only if every pseudoeffective class $\alpha$ can be written as a positive linear combination of the classes $\ell-\ell_i$ and $\ell_i$, for $1 \leq i \leq r$. More generally, we have the following criterion for a class $\alpha$ in $\ceff^m(X_\Gamma)$ to be effective.

\begin{lemma}\label{lem:sum_alambda}
Let $\alpha\in \ceff^m(X_\Gamma)$ be a class $\alpha = \sum_{|\lambda|=m}a_\lambda \sigma_\lambda - \sum_{i=1}^r b_i E_i^{[m]}$, $b_i \geq 0$ for all $0 \leq i \leq r$, such that
\[
	\sum_{|\lambda| = m}a_\lambda \geq \sum_{i=1}^r b_i.
\]
Then $\alpha$ is in the span of the classes $\sigma_\lambda - E_i^{[m]}$ and $E_i^{[m]}$ for $0 \leq i \leq r$.
\end{lemma}
\begin{proof}
We induct on $r$. In the case $r=1$ write $\alpha = \sum_\lambda a_\lambda \sigma_\lambda - bE^{[m]}$. Then
\[
\alpha = \sum_{\lambda}a_\lambda(\sigma_\lambda - E^{[m]}) + \left(-b+\sum_\lambda a_\lambda\right)E^{[m]}.
\]
Suppose the claim holds for $r-1$ points. Choose numbers $a'_\lambda$ such that $0 \leq a'_\lambda \leq a_\lambda$ and $\sum_\lambda a'_\lambda = b_r$. We have
\[
\alpha = \sum_{\lambda} (a_\lambda - a'_\lambda)\sigma_\lambda + \sum_{i=1}^{r-1} b_i E_i^{[m]} + \sum_{\lambda} a'_\lambda (\sigma_\lambda - E_r^{[m]}).
\]
Since $\sum_{\lambda} (a_\lambda - a'_\lambda) \geq \sum_{i=1}^{r-1} b_i$, the class $\sum_{\lambda} (a_\lambda - a'_\lambda)\sigma_\lambda + \sum_{i=1}^{r-1} b_i E_i^{[m]}$ is in the span of Schubert classes by the inductive hypothesis.
\end{proof}

\begin{prop}\label{prop:nef_implies_lingen}
If $\Gamma$ is a set of $r$ points on $\grass(k,n)$ then $\ceff_1(X_\Gamma)$ is S-generated if and only if the class $H - \sum_{i=1}^r E_i$ is nef.
\end{prop}
\begin{proof}
Let $\alpha = a \ell - \sum b_i \ell_i \in \ceff_1(X)$. By Lemma \ref{lemma:positive_coeffs}, we may assume $b_i \geq 0$ for $0 \leq i \leq r$. Since $\ell \cdot H = 1$ we have:
\[
	\alpha \cdot \left( H - \sum_i E_i \right) = a - \sum b_i \geq 0.
\]
Thus the intersection is nonnegative if and only if $a \geq \sum b_i$ if and only if $\alpha$ is in the span of the classes $\ell$, $\ell_i$, and $\ell - \ell_i$ by Lemma \ref{lem:sum_alambda}.
\end{proof}

We can now show that $\ceff_m(X_\Gamma)$ is S-generated for $m=1,2$ when $\Gamma$ consists of at most $\binom{n}{k}-k(n-k)$ general points. We will make use of a few standard facts from Schubert calculus. First, $N^2(\grass(k,n))$ is spanned by the two Schubert classes $\sigma_2$ and $\sigma_{1,1}$. Similarly, $N_2(\grass(k,n))$ is spanned by the classes $\sigma_2^\ast = \sigma_{n-k,n-k,\dots,n-k,n-k-2}$ and $\sigma_{1,1}^\ast = \sigma_{n-k,n-k,\dots,n-k,n-k-1,n-k-1}$. These classes satisfy
\[
\sigma_2^\ast \cdot \sigma_2 = \sigma_{1,1}^\ast \cdot \sigma_{1,1} = 1, \qquad \sigma_2^\ast \cdot \sigma_{1,1} = \sigma_{1,1}^\ast \cdot \sigma_2 = 0, \qquad \sigma_1^2 = \sigma_2 + \sigma_{1,1}.
\]

\begin{prop}\label{prop:lingen_terriblebound}
Let $\Gamma$ be a set of $r$ general points on $\grass(k,n)$ and suppose
\[
r \leq \binom{n}{k} - k(n-k).
\]
Then
\begin{enumerate}
\item $\ceff_1(X_\Gamma)$ is generated by the exceptional classes and classes of the form $\ell - \ell_i$ for $0 \leq i \leq r$.
\item $\ceff_2(X_\Gamma)$ is generated by the exceptional classes and classes of the form $\sigma_{2}^\ast - E_{i,[2]}$ and $\sigma_{1,1}^\ast - E_{i,[2]}$ for $0 \leq i \leq r$, where $\sigma_\lambda^\ast$ denotes the the Schubert class dual to $\sigma_\lambda$.
\end{enumerate}
In particular, $\ceff_1(X_\Gamma)$ and $\ceff_2(X_\Gamma)$ are S-generated.
\end{prop}
\begin{proof}
Let $\Gamma$ be a set of $r$ general points and $\Gamma' \subset \Gamma$ a set of $r'$ points. If $\alpha$ is an effective $m$-cycle in $\ceff_m(X_{\Gamma'})$ then the strict transform of $\alpha$ in $\ceff_m(X_\Gamma)$ is effective. Thus if the claim holds for $\Gamma$, it holds for $\Gamma'$. It therefore suffices to consider the case $r = \binom{n}{k} - k(n-k)$.

The base locus of the linear system of hyperplanes containing $\Gamma$ is the $(r-1)$-plane $\Lambda$ spanned by $\Gamma$. Since $\dim \Lambda + \dim \grass(k,n) = \binom{n}{k} -1$, the intersection $\Lambda \cap \grass(k,n)$ in $\P^{\binom{n}{k}-1}$ is a finite set. Thus the divisor class $H- \sum_{i=1}^r E_i$ is nef. The claim for $\ceff_1(X_\Gamma)$ follows from Proposition \ref{prop:nef_implies_lingen}. 

Let $\beta$ be a pseudoeffective 2-cycle on $X_\Gamma$. By  Lemma \ref{lemma:positive_coeffs}, we may assume $\beta$ is of the form:
\[
\beta = a_2\sigma_{2}^\ast +a_{1,1}\sigma_{1,1}^\ast - \sum_{i=1}^r b_i E_{i,[2]},
\]
with $b_i \geq 0$ for $0 \leq i \leq r$. Since $H-\sum_{i=1}^r E_i$ is nef, we have
\[
0 \leq \beta \cdot \left(H- \sum_{i=1}^r E_i \right)^2=a_2 + a_{1,1} - \sum_{i=1}^r b_i.
\]
By Lemma \ref{lem:sum_alambda}, $\beta$ is in the span of the Schubert classes.
\end{proof}

 There exist Grassmannians for which $\ceff_1(X_\Gamma)$ is not S-generated for $r > \binom{n}{k} - k(n-k)$. For example, we will show this for $\grass(2,4)$ in Proposition \ref{prop:curves_g24}. However, Proposition \ref{prop:lingen_terriblebound} can typically be improved by one, at least for curves.
\begin{prop}\label{prop:lingen_prettybadbound}
Suppose $k$ and $n$ are such that $d \geq \binom{n}{k} - k(n-k) +1$, where $d$ is the degree of the Grassmannian in its Pl\"ucker embedding. Then $\ceff_1(X_\Gamma)$ is S-generated if
\[
r \leq \binom{n}{k} -k(n-k) + 1.
\]
\end{prop}
\begin{proof}
Following the proof of Proposition \ref{prop:lingen_terriblebound}, we examine the linear system of hyperplanes containing the points $p_i$. The base locus is the linear space $\Lambda$ spanned by the $p_i$ of $\Gamma$. Taking $r = \binom{n}{k} + k^2 - kn + 1$, we have $\dim(\Lambda \cap \grass(k,n))=1.$ The only curves that can meet the class $H - \sum_{i=1}^r E_i$ negatively are those of class $d\ell - \sum_{i=1}^r \ell_i$. But $d\geq r$ by assumption, so $\ceff_1(X)$ is S-generated by Propositions \ref{lem:sum_alambda} and \ref{prop:nef_implies_lingen}.
\end{proof}

\begin{remark}
The degree of the Pl\"ucker embedding is
\[
d=(k(n-k))! \prod_{i=1}^k \frac{(i-1)!}{(n-k+i-1)!}
\]
(see, for example, \cite{degree-grassmannians}).
\end{remark}

If we weaken the hypotheses of Proposition \ref{prop:lingen_terriblebound} to allow for very general points, we can characterize precisely when the cone of pseudoeffective curves is S-generated.

\begin{prop}\label{prop:picrank1_verygeneral}
Let $X \subsetneq \P^n$ be a nondegenerate irreducible variety of degree $d$ such that $X$ is covered by lines and has Picard rank 1. Let $\pi:X_r\to X$ be the blow-up of $X$ at $r$ very general points. If $r \leq d$ then $\ceff_1(X_r)$ is generated by exceptional classes and classes of proper transforms of lines.
\end{prop}
\begin{proof}
It suffices to consider the case $r=d$. Let $\Lambda$ be a general linear space of dimension $n-m$, and $\{p_1, \dots, p_r\} = X \cap \Lambda$. The linear system of hyperplanes on $X$ containing each $p_i$ has precisely the set $\{p_1, \dots, p_r\}$ as its base locus. Let $X_r$ be the blow-up of $X$ at these points and let $H$ be the effective generator of $N^1(X)$. Then $\pi^\ast H - \sum_{i=1}^r E_i \in \ceff^1(X_r)$ is nef.

Let $\ell \in \ceff_1(X_r)$ be the class of the proper transform of a line not passing through any of the $p_i$. Then $\pi^\ast H \cdot \ell = 1$ and $E_i \cdot \ell = 0$ for $i=1, \dots, r$. Let $\alpha \in \ceff_1(X_r)$ have class $a\ell - \sum_{i=1}^r b_i E_i^{[m-1]}$. Since $H- \sum_{i=1}^r E_i$ is nef we must have $a \geq \sum_{i=1}^r b_i$.
Since $X$ is covered by lines, it follows that $\ceff_1(X_r)$ is generated by the classes of strict transforms of lines and the exceptional classes. By \cite[Prop. 1.4.14]{lazarsfeld}, the same must hold for a very general set of points $p_1, \dots, p_r$.
\end{proof}

The one-dimensional Schubert variety $\Sigma_{n-k, \dots, n-k, n-k-1}$ maps to a line under the Pl\"ucker embedding. We obtain the following corollary.

\begin{cor}\label{prop:lingen_verygeneral} Let $X_r$ be the blow-up of $\grass(k,n)$ at $r$ very general points. Then $\ceff_1(X_r)$ is S-generated if and only if $r \leq d$, where $d$ is the degree of $\grass(k,n)$ in the Pl\"ucker embedding.
\end{cor}
\begin{proof}
Since $\grass(k,n)$ has Picard rank 1 and is covered by lines, Proposition \ref{prop:picrank1_verygeneral} shows that $\ceff_1(X_r)$ is S-generated for $r \leq d$. Conversely, suppose $r>d$. Then the linear system $H - \sum_{i=1}^r E_i$ satisfies $(H - \sum_{i=1}^r E_i)^m<0$, where $m=k(n-k) = \dim \grass(k,n)$. In particular, it is not nef. It therefore has negative intersection with some curve class $\alpha = a\ell - \sum_{i=1}^r b_i \ell_i$. The intersection number is $a - \sum b_i$, thus $\ceff_1(X_r)$ cannot be S-generated.
\end{proof}

\section{Examples: $\grass(2,4)$ and $\grass(2,5)$}\label{sec:examples}
Our first example is the pseudoeffective cone of blow-ups of the Grassmannian $\grass(2,4)$ of lines in $\P^3$. We will continue to denote by $\ell$ the class of the strict transform of the Schubert variety $\Sigma_{2,1}$, by $E_i$ the exceptional divisors, and by $\ell_i$ the one-dimensional linear classes contained in them. Let $r$ be a non-negative integer. If $r\leq 2$, then Theorem \ref{thm:divisors_2pts} implies that $\ceff^1(X_\Gamma)$ is S-generated and Proposition \ref{prop:lingen_terriblebound} implies the same for $\ceff_1(X_\Gamma)$ and $\ceff_2(X_\Gamma)$. Under the Pl\"ucker embedding, $\grass(2,4)$ is a quadric hypersurface in $\P^5$. A calculation for quadric hypersurfaces allows us to compute the cone of pseudoeffective curves on blow-ups of $\grass(2,4)$. The following lemma was shown independently in \cite{rischter}.

\begin{lemma}\label{lemma:curves_on_quadrics}
Let $Q \subset \P^{n+1}$ be a quadric hypersurface of dimension $n \geq 3$. Let $\Gamma$ be a set of $r$ general points on $Q$ and $Q_\Gamma$ the blow-up of $Q$ along $\Gamma$. We have the following:
\begin{enumerate}
	\item If $0 \leq r \leq 2$ then $\ceff_1(Q_\Gamma)$ is generated by the exceptional classes and the classes of strict transforms of lines on $Q$.
	\item If $3 \leq r \leq 6$ then $\ceff_1(Q_\Gamma)$ is generated by the exceptional classes and the classes of strict transforms of lines and conics on $Q$.
\end{enumerate}
\end{lemma}
\begin{proof}
Since $n\geq 3$, $\pic Q$ has rank 1, thus $N_1(Q)$ and $N^1(Q)$ are 1-dimensional and the generator of $N_1(Q)$ is the class of a line. Let $H \in N^1(\P^{n+1})$ be the class of a hyperplane. Then in the blow-up $\P^{n+1}_\Gamma$ of $\P^{n+1}$ along $\Gamma$, the base locus of the class $\pi^\ast H - E_1 - E_2$ is the line spanned by the two points of $\Gamma$. In particular, its restriction to $Q_\Gamma$ is nef. Let $\ell\in N_1(Q_\Gamma)$ be the pullback of the class of a line and $\ell_i \in N_1(Q)$ the class of a line contained in the exceptional divisor $E_i$. Then if $a\ell - b_1\ell_1 - b_2\ell_2$ is effective, we must have $a \geq b_1 + b_2$. Since $Q$ is covered by lines, the classes $\ell - \ell_i$ are effective. Since $a \geq b_1 + b_2$, the classes $\ell-\ell_i$ and $\ell_i$ must generate the effective cone.

Suppose $3 \leq r \leq 6$. We project $Q \subset \P^{n+1} \to \P^{n-1}$ from two points $p_i$ and $p_j$ of $\Gamma$ and take cones over quadric surfaces containing the remaining $p_k$. Note that the vector space of quadric hypersurfaces in $\P^{n-1}$ has dimension $\frac{1}{2}n(n+1)$. Since $n \geq 3$ and $r \leq 6$, the base locus of quadrics $Q' \subset \P^{n-1}$ containing the $p_k$ is just the set of the $p_k$. By taking cones over these $Q'$, singular along the line $\overline{p_ip_j}$, we obtain quadrics in $\P^{n+1}$ with double points at $p_i$ and $p_j$. The base locus of such quadrics is thus the set of planes $\overline{p_ip_jp_k}$ as $k$ varies.

The intersection of such a plane with $Q$ is a conic whose strict transform in $Q_\Gamma$ has class $2\ell - \ell_i - \ell_j - \ell_k$. If $C \in N_1(Q_\Gamma)$ has class $a\ell - \sum_{i=1}^6 b_i\ell_i$ and $C$ is not contained in one of these planes, then $C$ has proper intersection with the divisor $2\pi^\ast H - 2E_i-2E_j - \sum_{k \neq i,j} E_k$. That is, $C$ either contains a conic of the form $2\ell - \ell_i - \ell_j - \ell_k$ or it satisfies the inequality

\begin{equation}\label{eqn:inequality_b}
	2a \geq 2b_i + 2b_j + \sum_{k\neq i,j} b_k.
\end{equation}

After subtracting off components of $C$ that are conics, we may assume (\ref{eqn:inequality_b}) holds for all $i,j$. Suppose $b_1$ and $b_2$ are maximal among the coefficients of the $\ell_k$ in the expression of $C$. Fix a choice of indices $i,j,k$. We subtract $2\ell - \ell_i - \ell_j - \ell_k$ from $C$ a total of $\left\lfloor \frac{a}{2} \right\rfloor$ times to obtain a class $D = a'\ell - \sum b_i' \ell_i$. We claim $D$ is in the span of strict transforms of lines. Assuming this, $C$ is then the sum of $D$ and conics, proving the claim. Suppose first that $a$ is even. Then
\[
\sum_{i=1}^r b_i' = \sum_{i=1}^r b_i - \frac{3a}{2} = 2b_1 + 2b_2 + \sum_{i=3}^r b_i - \frac{3a}{2} - b_1 -b_2.
\]
It suffices to show that $\sum b_i' \leq 0$. Applying (\ref{eqn:inequality_b}), we have
\[
\sum_{i=1}^r b_i' \leq 2a - \frac{3a}{2} - b_1 - b_2.
\]
We may assume that $b_1$ and $b_2$ are as small as possible. Since they are assumed to be maximal among the $b_i$, it suffices to prove the claim when $b_1 = b_2 = \cdots = b_r$. In this case, $2a \geq (r+2)b_1$ by (\ref{eqn:inequality_b}), whence:
\[
2a - \frac{3a}{2} - b_1 - b_2 \leq 2a - \frac{3a}{2} - \frac{4a}{r+2} = \frac{a(r-6)}{2(r+2)}.
\]
Since $r \leq 6$, the it follows that $\sum b_i' \leq 0$.

If $a$ is odd then $\left\lfloor \frac{a}{2} \right\rfloor = \frac{a-1}{2}$ and it suffices to show that $\sum b_i' \leq 1$. We have:
\[
\sum_{i=1}^r b_i' \leq 2a - \left(\frac{3a-3}{2}\right) - b_1 - b_2.
\]
Again we may assume $b_1 = b_2 = \cdots = b_r$, so that
\[
2a - \left(\frac{3a-3}{2}\right) - b_1 - b_2 \leq \frac{a+3}{2} - \frac{4a}{r+2} = \frac{a(r-6) + 3(r+2)}{2(r+2)} < 2.
\]
Since $\sum b_i'$ is an integer, it follows that $\sum b_i' \leq 1$.
\end{proof}

\begin{prop}\label{prop:curves_g24}
Let $\Gamma$ be a set of $r$ general points on $\grass(2,4)$, and $X_\Gamma$ the blow-up of $\grass(2,4)$ along $\Gamma$.
\begin{enumerate}
	\item For $r=1$ and $r=2$, $\ceff_1(X_\Gamma)$ is S-generated.
	\item For $3 \leq r \leq 7$, $\ceff_1(X_\Gamma)$ is generated by the Schubert classes, the exceptional classes, and by conics of the form $2\ell - \ell_i - \ell_j - \ell_j$, for $i,j,k$ distinct.
\end{enumerate}
\end{prop}
\begin{proof}
By Lemma \ref{lemma:curves_on_quadrics}, it suffices to resolve the case $r=7$. Consider the linear system $H - E_1 - E_2 - E_3-E_4 - E_5$ of hyperplanes through five of the points $p_1, \dots, p_5$. Its base locus is the intersection of the $\P^4$ they span with $\grass(2,4)$, hence a quadric 3-fold $Y$. By Lemma \ref{lemma:curves_on_quadrics}, the pseudoeffective cone $\ceff_1(Y_5)$ of the blow-up of $Y$ at the five points $p_1, \dots, p_5$ is generated by lines and conics. Therefore on $X_\Gamma$ a pseudoeffective curve class either lies on a quadric threefold and is therefore in the span of lines and conics, or it meets the class $H - \sum_{j=1}^5 E_{i_j}$ positively and therefore satisfies the following inequality:
\begin{equation*}\label{eqn:inequality_c}
 a\geq b_{i_1} + \cdots + b_{i_5}.
\end{equation*}
Let $\alpha$ be a pseudoeffective cycle and subtract $2\ell - \ell_i - \ell_j - \ell_k$ from $\alpha$ a total of $\left\lfloor \frac{a}{2} \right\rfloor$ times. If $a$ is even, it suffices to show that $\sum_{i=1}^7 b_i -\frac{3a}{2}\leq 0$. Without loss of generality, assume $b_1 \geq b_2 \geq \cdots \geq b_7.$ By the above inequality, we have:
\[
\sum_{i=1}^7 b_i - \frac{3a}{2} \leq b_6 + b_7 - \frac{a}{2}.
\]
Suppose now that $b_6 + b_7 > \frac{a}{2}$. Since $(b_6+b_7) + b_1 + b_2 + b_3 \leq a$, it follows that at least one of $b_1$, $b_2$ and $b_3$ is at most $\frac{a}{6}$. But $b_6$ is at least $\frac{a}{4}$, contradicting its  minimality. Thus $\sum_{i=1}^7 b_i - \frac{3a}{2} \leq 0$.

Similarly, if $a$ is odd it suffices to show that $\sum_{i=1}^7 b_i - \frac{3(a-1)}{2} \leq 1$. If not, then $b_6 + b_7 > \frac{a-1}{2}$. Again, we must have $b_i \leq \frac{a}{6}$ for some $i \neq 6,7$, but $b_6 \geq \frac{a-1}{4}$. This is impossible if $a \geq 1$, so we again arrive at a contradiction.
\end{proof}

Proposition \ref{prop:lingen_terriblebound} implies that if $\Gamma$ consists of at most two general points, then $\ceff_2(X_\Gamma)$ is S-generated. Using the rigidity of the 2-dimensional Schubert classes, we show below that the converse is also true. Note that Theorem \ref{thm:divisors_2pts} implies $\ceff^1(X_\Gamma)$ is S-generated when $r=2$ and so we have that $\ceff^m(X_\Gamma)$ is S-generated for $r \leq 2$ for all $1 \leq m \leq 3.$

\begin{prop}\label{prop:planes_g24}
Let $X_\Gamma$ be the blow-up of $\grass(2,4)$ at a set $\Gamma$ of $r$ general points. Then the pseudoeffective cone $\ceff_2(X_\Gamma)$ is S-generated if and only if $r \leq 2$.
\end{prop}
\begin{proof}
That $\ceff_2(X_\Gamma)$ is S-generated for $r\leq 2$ follows immediately from Proposition \ref{prop:lingen_terriblebound}. Conversely, let $\Gamma$ be a set of three general points on $\grass(2,4)$. If $\Lambda\subset \P^5$ is a general 3-plane containing $\Gamma$, then $\Lambda \cap \grass(2,4)$ is a quadric surface of class $\sigma_{1,1} + \sigma_2$. Thus the class $\sigma_{1,1} + \sigma_2 - E_{1,[2]} - E_{2,[2]}- E_{3,[2]}$ is effective. A Schubert variety of class $\sigma_{1,1}$ or $\sigma_2$ can pass through only one general point. Indeed, $\sigma_{1,1}$ corresponds to the class in $\grass(2,4)$ of planes contained in a fixed 3-dimensional vector space and $\sigma_{2}$ corresponds to the class of planes containing a fixed line. Further, the classes $\sigma_{1,1}$ and $\sigma_2$ are rigid by \cite{coskun-rigidity}, thus any variety of such a class can pass through only one general point. Thus $\sigma_{1,1} + \sigma_2 - E_{1,[2]} - E_{2,[2]} - E_{3,[2]}$ cannot be in the span of Schubert classes.
\end{proof}

Our next example is the Grassmannian $\grass(2,5)$ of lines in $\P^4$. Under the Pl\"ucker embedding, $\grass(2,5)$ is a quintic 6-fold in $\P^9$. Let $\Gamma$ be a set of $r$ points on $\grass(2,5)$ and $X_\Gamma$ the blow-up of $\grass(2,5)$ along $\Gamma$. By Corollary \ref{cor:lin_gen1pt}, $\ceff^1(X_\Gamma)$ is S-generated if $r = 1$. By Propositions \ref{prop:lingen_prettybadbound} and \ref{prop:lingen_terriblebound}, $\ceff_1(X_\Gamma)$ is S-generated for $r \leq 5$ and $\ceff_2(X_\Gamma)$ is S-generated for $r \leq 4$. We will show in Section \ref{sec:infinitely_gen} that if $r=6$, $\ceff_1(X_\Gamma)$ is not finitely generated (assuming the SHGH conjecture). We conclude this section by showing that $\ceff_3(X_\Gamma)$ is S-generated when $\Gamma$ consists of $r \leq 4$ general points.

\begin{prop}
Let $X_\Gamma$ be the blow-up of $\grass(2,5)$ at $r$ general points. If $r \leq 4$ then $\ceff_3(X_\Gamma)$ is S-generated.
\end{prop}
\begin{proof}
Suppose first that $r=1$ and $\alpha = a_{2,1} \sigma_{2,1} + a_3 \sigma_3 - b E^{[3]}$ is a 3-cycle. The linear system $H-E$ is nef, so we have
\[
0 \leq (H-E)^3 \cdot \alpha = (2 \sigma_{2,1} + \sigma_3 - E^{[3]})\cdot \alpha = 2a_{2,1} + a_3.
\]
A Schubert variety of class $\sigma_{2,1}$ has a singularity of multiplicity 2, thus $\sigma_{2,1} - 2E^{[3]}$ is an effective class. If $b \geq 2a_{2,1}$ then we subtract $a_{2,1}(\sigma_{2,1} - 2E^{[3]})$ from $\alpha$ to obtain an effective class
\[
    a_3\sigma_3 - (b-2a_{2,1})E^{[3]}.
\]
Since $a_3 \geq b-2a_{2,1}$, we may write
\[
\alpha = a_{2,1}(\sigma_{2,1} - 2E^{[3]}) + (a_3+2a_{2,1}-b)(\sigma_3-E^{[3]}).
\]
Suppose now $2 \leq r \leq 4$. Take a pseudoeffective class
\[
\alpha = a_{2,1}\sigma_{2,1} + a_3\sigma_{3} - \sum_{i=1}^r b_i E_i^{[3]}.
\]
By the nefness of $H-\sum E_i$ we have the inequality
\begin{equation}\label{inequality:g25}
2a_{2,1} + a_3 \geq \sum b_i.
\end{equation}
Now $\Sigma_{2,1}$ is the variety of 2-dimensional vector spaces contained in a 4-dimensional space. Fix two general points $W, W' \in \grass(2,5)$. Let $V = W+W'$ be the 4-dimensional space they span, and $F_\bullet$ a flag such that $F_4 = V$. Then the Schubert variety $\Sigma_{2,1}(F_\bullet)$ contains $W$ and $W'$. Thus there exist Schubert varieties of class $\sigma_{2,1}$ through 2 general points and Schubert varieties of class $\sigma_3$ through 1 general point. Schubert varieties of class $\sigma_{2,1}$ have singularities of multiplicity 2.

We will show by induction on $r$ that any class satisfying (\ref{inequality:g25}) must be in the span of the Schubert classes. Without loss of generality, assume $b_1 \geq b_2 \geq \cdots \geq b_r.$ If $2a_{2,1} \geq b_1$ then subtract $\left\lfloor \frac{b_1}{2} \right\rfloor(\sigma_{2,1} - 2E^{[3]})$ from $\alpha$. If $b_1$ is even then we have
\[
\alpha ' := \alpha - \left\lfloor \frac{b_1}{2} \right\rfloor(\sigma_{2,1} - 2E^{[3]}) = \left(a_{2,1} - \left\lfloor \frac{b_1}{2} \right\rfloor\right)\sigma_{2,1} + a_3\sigma_3 - \sum_{i=2}^{r} b_i E_i^{[3]}
\]
By the induction hypothesis, we are done. If $b_1$ is odd then further subtract $\sigma_{2,1} - E_1^{[3]} - E_2^{[3]}$ to obtain
\[
\alpha '' := \alpha - \left\lfloor \frac{b_1}{2} \right\rfloor(\sigma_{2,1} - 2E^{[3]}) = \left(a_{2,1} - \left\lfloor \frac{b_1}{2} \right\rfloor - 1\right)\sigma_{2,1} + a_3\sigma_3 -(b_2-1)E_2^{[3]} - \sum_{i=3}^{r} b_i E_i^{[3]}.
\]
If $2a_{2,1} < b_1$ then subtract $a_{2,1}(\sigma_{2,1} - 2E_1^{[3]})$ from $\alpha$ to obtain
\[
\alpha ' := a_3\sigma_3 - (b_1 - 2a_{2,1})E_1^{[3]} - \sum_{i=2}^r b_iE_i^{[3]}.
\]
Since $a_3 \geq \sum b_i - 2a_{2,1}$, we can subtract off $\sigma_3 - E_i^{[3]}$ as needed.
\end{proof}

\section{Curves on Fano manifolds of index $n-1$}\label{sec:infinitely_gen}
The goal of this section is to show that the cone of pseudoeffective curves on blow-ups of certain Fano varieties is not finitely generated. This was proved by Coskun-Lesieutre-Ottem \cite{clo} for blow-ups of $\P^3$ at 9 or more very general points, assuming the SHGH conjecture. We will work over an algebraically closed field of characteristic zero.

\begin{definition}\label{def:fanomfld}
A smooth projective variety $Z$ of dimension $n\geq 3$ is called a \emph{Fano manifold of index} $n-1$ (or \emph{del Pezzo manifold}) if 
\[
\max \{m \in \Z_{>0} : -K_Z = mH \text{ for some }H \text{ ample}\} = n-1.
\]
\end{definition}

More generally, the \emph{index} of a Fano variety $Z$ is the largest integer $m$ such that $-K_Z = mH$ for an ample divisor $H$. If $\dim Z = n$ and $d=H^n$ is the degree of $Z$, then the number $\Delta = n - d - 1$ is called the \emph{$\Delta$-genus of $Z$}. Fano manifolds of index $n-1$ are therefore sometimes called \emph{varieties of $\Delta$-genus one} (see \cite{fujita-classification2}, \cite{fujita-structure}). Fano manifolds of index $n-1$ were classified by Fujita in \cite{fujita-classification2} and \cite{fujita-classification}, and by Iskovskih in \cite{iskovskih-3foldsi} and \cite{iskovskih-3foldsii}. If $Z$ is a Fano manifold of index $n-1$, then $Z$ is one of the following:
\begin{enumerate}[(1)]
	\item A linear section of the Grassmannian $\grass(2,5)$ in its Pl\"ucker embedding;
	\item A complete intersection of two quadric hypersurfaces in $\P^{n+2}$;
	\item A cubic hypersurface in $\P^{n+1}$;
	\item $\P^2 \times \P^2$;
	\item $F(1,2;3)$, the flag variety parameterizing complete flags in a 3-dimensional vector space;
	\item $\P^1 \times \P^1 \times \P^1$;
	\item A double cover $Z \to \P^n$ branched along a quartic hypersurface;
	\item The blow-up of $\P^3$ at a point;
	\item A sextic hypersurface in the weighted projective space $\P(3,2,1,\dots,1)$.
\end{enumerate}

Fano manifolds of index $n-1$ always contain a del Pezzo surface. Indeed, suppose $H$ is an ample divisor on $Z$ such that $-K_Z = (n-1)H$. Then by adjunction, the complete intersection $H^{n-2}$ is a surface with ample anticanonical. The degree of this del Pezzo surface depends on the structure of $Z$.

\begin{prop}
Let $Z$ be a Fano manifold of index $n-1$ with $-K_Z = (n-1)H$ for $H$ ample, and let $d=H^n$. If $\Gamma$ is a set of $d+1$ very general points on $Z$ then the base locus of the linear system of divisors $H$ containing $\Gamma$ is a del Pezzo surface of degree $d$.
\end{prop}
\begin{proof}
By \cite{fujita-structure}, $h^0(Z,H) = n+d - 1$. If we take $d+1$ very general points in $X$, then the linear system of divisors of class $H$ passing through each point has dimension $h^0(Z,H) - (d+1) = (n+d-1) - (d+1) = n-2$. Taking $n-2$ elements of $|H|$, their intersection has dimension 2 and class $H^{n-2}$. The canonical class of this surface $S$ is
\[
	K_S = (-(n-1)H + (n-2)H)|_S = -H|_S.
\]
Thus $-K_S$ is ample and so $S$ is a del Pezzo surface of degree $d$.
\end{proof}

The del Pezzo surface of degree $d$ is isomorphic to the blow-up of $\P^2$ at $9-d$ general points. If we choose $d+1$ very general points on $Z$ as in the proposition, the blow-up $X$ of $Z$ at those points contains the blow-up of $\P^2$ at 10 very general points. Using this fact and the following conjecture, we will show that the cone of pseudoeffective curves on $X$ is not finitely generated.

\begin{conj*}[Segre-Harbourne-Gimigliano-Hirschowitz (SHGH) Conjecture] Let $Y$ be the blow-up of $r$ very general points of $\P^2$ and let $\mathcal{L}$ be a linear system on $Y$. Then $h^0(\mathcal{L}) \cdot h^1(\mathcal{L}) > 0$ if and only if there exists a $(-1)$-curve $E$ such that $E \cdot C < 0$, where $C$ is a general curve of $\mathcal{L}$.
\end{conj*}

The following theorem of de Fernex shows that the SHGH conjecture describes the nef cone on the blow-up of $\P^2$.

\begin{thm}[de Fernex \cite{de-fernex}]
Assume the SHGH conjecture. Let $Y$ be the blow-up of $\P^2$ at $r \geq 10$ points and $P \subset N^1(Y)$ the cone
\[
	P = \{ D \in N^1(Y) : D^2 \geq 0, D \cdot H \geq 0\}
\]
for $H$ any ample divisor. Then
\[
\ceff_1(Y) \cap K_{\geq 0} = P \cap K_{\geq 0},
\]
where $K_{\geq 0} = \{C \in N^1(Y) : C \cdot K_Y \geq 0\}$.
\end{thm}

The proof of Theorem \ref{thm:fanomfld_infinite} rests on our ability to show that certain curve classes are nef on the blow-up of $\P^2$. For this reason, it seems that the SHGH conjecture is indispensable. The purpose of the next lemma is to show that we can push certain classes forward from the blow-up of $\P^2$ to the blow-up of a Fano manifold of index $n-1$ and obtain an extremal ray of the pseudoeffective cone.

\begin{lemma}[\cite{clo}]\label{lemma:extremal_curves}
Let $S\subseteq X$ be a surface contained in a projective variety $X$ such that the map $i_\ast:N_1(S) \to N_1(X)$ induced by inclusion is surjective. Then we may write
\[
\ceff_1(X) = i_\ast\ceff_1(S) + \Sigma
\]
for some cone $\Sigma$. Let $\Gamma \subset N_1(S)$ be a cone such that $i_\ast \Gamma = \Sigma$. Suppose $D \in N_1(X)$ is a class satisfying
\begin{enumerate}[(1)]
	\item $D$ is nef, and if $D \cdot \gamma = 0$ for $\gamma \in \ceff_1(S)$, then $\gamma$ is a multiple of $D$;
	\item $D\cdot C = 0$ for all $C$ in the kernel of $i_\ast$;
	\item $D \cdot C >0$ for all $C$ in $\Gamma$.
\end{enumerate}
Then $i_\ast D$ is nonzero and spans an extremal ray in $\ceff_1(X)$.
\end{lemma}
\begin{proof}
Let $\Omega\subset N_1(S)$ be the cone
\[
\Omega = \ceff_1(S) + \ker i_\ast + \Gamma.
\]
Suppose $D = \alpha + \beta$ for $\alpha, \beta \in \Omega$. Conditions (1)--(3) imply that $D$ is in the dual cone of $\Omega$ and $D^2=0$, thus $D\cdot \alpha = D \cdot \beta = 0$. By (1) and (2) we see that $\alpha = a_1D + C_1$, $\beta = a_2D + C_2$, where $C_1, C_2 \in \ker i_\ast$.

Suppose now $i_\ast D = \alpha+\beta$ for $\alpha,\beta \in \ceff_1(X)$. Then we may write $\alpha = i_\ast \alpha_S + \sigma_\alpha$, $\beta = i_\ast \beta_S + \sigma_\beta$ for some $\alpha_S,\beta_S \in \ceff_1(S)$ and $\sigma_\alpha, \sigma_\beta \in \Sigma$. Let $\gamma_\alpha, \gamma_\beta \in \Gamma$ be such that $i_\ast \gamma_\alpha =  \sigma_\alpha$ and $i_\ast \gamma_\beta = \sigma_\beta$. Then we have
\[
i_\ast D = i_\ast (\alpha_S + \beta_S + \gamma_\alpha + \gamma_\beta),
\]
so $D = \alpha_S + \beta_S + \gamma_\alpha + \gamma_\beta + C$, for some $C \in \ker i_\ast$. Since $D^2=0$, conditions (1)--(3) imply $\gamma_\alpha=\gamma_\beta = 0$. It follows that $\sigma_\alpha = \sigma_\beta = 0$ and thus $D = i_\ast \alpha_S + i_\ast\beta_S$. By the above, we may write $\alpha_S = a_1D + C_1$, $\beta_S = a_2D + C_2$ for some $C_1,C_2 \in \ker i_\ast$. Then $\alpha = i_\ast \alpha_S = a_1 i_\ast D$ and $\beta = i_\ast \beta_S = a_2 i_\ast D$. Thus $D$ spans an extremal ray.
\end{proof}

Note that we may always take $\Gamma$ to be the cone $i_\ast^{-1}(\Sigma) \cap \ceff_1(X)$.

\begin{thm}\label{thm:fanomfld_infinite}
Let $X$ be the blow-up of a Fano manifold $Z$ of index $n-1$ at $d+1$ very general points. Assuming the SHGH conjecture holds for the blow-up of $\P^2$ at 10 very general points, there exist infinitely many classes $D$ satisfying the hypotheses of Lemma \ref{lemma:extremal_curves}.
\end{thm}
\begin{proof}
Let $\pi:S\to \P^2$ be the blow-up of $\P^2$ at 10 very general points. We have seen that $S \subset X$. The surface $S$ contains 10 exceptional classes, $d+1$ of which lie over points blown-up on $Z$. We denote these classes by $f_1, \dots, f_{d+1} \in N_1(S)$. The remaining $\pi$-exceptional classes will be denoted $e_1, \dots, e_N$, where $N=9-d$. Let $h$ be the pullback of the hyperplane class in $\P^2$ via $\pi$. Then $\{h, e_1, \dots, e_N, f_1, \dots, f_{10-N}\}$ is a basis for $N_1(S)$.

Let $\delta$ be a real number such that $\sqrt{\frac{8-N}{9(9-N)}} < \delta < \frac{1}{3}$. Observe that for $N\leq 8$ this range is nonempty. Let $\delta' = \sqrt{(9-N)\left(\frac{1}{9}-\delta^2\right)}$. It follows that $0 < \delta,\delta' < \frac{1}{3}.$ Define
\[
D_\delta = h - \frac{1}{3}\sum_{i=1}^N e_i - \delta' f_1 - \delta \sum_{j=2}^{10-N} f_j.
\]
We show that $D$ is nef. We have
\[
D^2 = 1 - \frac{N}{9} - \delta'^2 - \frac{9-N}\delta^2 = 0.
\]
Suppose now that $C \in \ceff_1(S)$ and $K_S \cdot C < 0$. If $C=e_i$ or $C=f_j$ then $D\cdot C > 0$. We may therefore write $C = ah - \sum_i b_i e_i - \sum_j c_j f_j$ with $b_i,c_j \geq 0$. Since $K_S = -3h + \sum_i e_i + \sum_j f_j$ it follows that $3a > \sum_i b_i + \sum_j c_j$. Then
\[
3 (D \cdot C) = 3a - \sum_{i=1}^N b_i - 3 \delta' c_1 - 3\delta \sum_{j=2}^{10-N} c_j.
\]
Since $\delta, \delta' < \frac{1}{3}$, this number is positive. Suppose now $K_S \cdot C \geq 0$. By the SHGH conjecture, it suffices to show that $D \cdot C \geq 0$ when $C^2 \geq 0$. If $C^2 \geq 0$ then we have $a^2 \geq \sum_i b_i^2 + \sum_j c_j^2$. Because $D^2 =0$, it follows by the Cauchy-Schwartz inequality that $D \cdot C \geq 0$, and equality holds if and only if $C$ is a multiple of $D$. This gives (1).

Now we must consider the structure of $X$. Note that if the claim holds for $\grass(2,5)$ then it holds for any linear section of $\grass(2,5)$ of dimension at least 3. Similarly, if the claim holds for $\P^2 \times \P^2$ then it holds for $F(1,2;3)$ which is a hyperplane section. We begin with case (1): $Z = \grass(2,5)$. We have $-K_Z = 5\sigma_1$ and the intersection of hyperplanes containing 6 very general points is a del Pezzo surface of degree 5, isomorphic to the blow-up of $\P^2$ at 4 points $\{q_1, \dots, q_4\}$.

The embedding of the degree 5 del Pezzo surface into $\P^5$ is given by the linear system of cubics in $\P^2$ vanishing at the four points $q_1, \dots, q_4$. If $\ell$ is the class of a line in $X$, then the map $i_\ast$ is defined by
\[
i_\ast h = 3\ell \qquad i_\ast e_i = \ell \qquad i_\ast f_j = \hat{f}_j,
\]
where $\hat{f}_j\in N_1(X)$ is the class of a line contained in the exceptional divisor over $p_j$ in $X$. The kernel of $i_\ast$ is clearly generated by the classes $h - 3e_i$. Further, we have
\[
\ceff_1(X) = i_\ast \ceff_1(S) + \sum_{j=1}^{6} \R_{\geq 0} (\ell - \hat{f}_j).
\]
Indeed, the linear system $H-\sum_{j=1}^{6}E_j$ must have non-negative intersection with any class not contained in its base locus. The base locus is the surface $S$, and if $C$ is a curve whose intersection with this linear system is non-negative then $C$ is in the span of the linear classes (c.f. Proposition \ref{prop:picrank1_verygeneral}). In the notation of the lemma, we have $\Sigma = \sum_{j=1}^{6} \R_{\geq 0} (\ell - \hat{f}_j)$ and we take $\Gamma$ to be the positive cone spanned by $(h - 3f_j)$ for $1 \leq j \leq 6$.

The intersection numbers are immediate: $D \cdot (h - 3e_i) = 0$ and $D \cdot (h-3f_j) = 1-3\delta$ if $j\geq 2$ and $D \cdot (h-3f_2) = 1- 3\delta'$. Since $\delta, \delta' < \frac{1}{3}$, these intersection numbers are strictly positive. This gives (2) and (3) and completes the proof for $Z = \grass(2,5)$.

The proofs for the remaining cases proceed in a similar manner. We summarize the relevant data in the following table.

\begin{center}
\def\arraystretch{2.0}
\begin{tabular}{|c|c|c|c|}
\hline
$Z$ & Degree & $\ker i_\ast$ & $\Gamma$\\
\hline\hline
$\grass(2,5)$ & 5 & $\sum_{i=1}^4\R_{\geq 0}(h-3e_i)$ & $\sum_{j=1}^{6} (h-3f_j)$\\
\hline
$Q_1 \cap Q_2$& 4 & $\sum_{i=1}^5\R_{\geq 0}(h-3e_i)$ & $\sum_{j=1}^{5} (h-3f_j)$\\
\hline
Cubic hypersurface & 3 & $\sum_{i=1}^6\R_{\geq 0}(h-3e_i)$ & $\sum_{j=1}^{4} (h-3f_j)$\\
\hline
$\P^2 \times \P^2$ & 6 & $\sum_{i,j}\R_{\geq 0}(e_i-e_j)$ & $\sum_{i,j} \R_{\geq 0}(e_i-f_j) + \sum_{i,j,k}\R_{\geq 0}(h-e_i-e_j-f_k)$\\
\hline
$\P^1 \times \P^1 \times \P^1$ & 6 & $\R_{\geq 0} (h-e_1-e_2-e_3)$ & $\sum_{i,k}(e_i - f_k)$\\
\hline
Double cover of $\P^n$ & 2 & $\sum_{i=1}^7 \R_{\geq 0} (h-3e_i)$ & $\sum_{j=1}^{3} \R_{\geq 0}(h-3f_j)$\\
\hline
{\parbox{3cm}{Sextic hypersurface in $\P(3,2,1,\dots,1)$}} & 1 & $\sum_{i=1}^8 \R_{\geq 0} (h-3e_i)$ & $\sum_{j=2}^{4} (h-3f_j)$\\
\hline
{\parbox{3cm}{Blow-up of $\P^3$ at a point}} & 7 & $\sum_{i=1}^2 \R_{\geq 0} (h - 3e_i)$ & $\sum_{j=1}^8 \R_{\geq 0}(h - 3f_j)$\\
\hline
\end{tabular}
\end{center}
\end{proof}

Returning to the case $Z= \grass(2,5)$ we can show that blow-ups of certain Grassmannians at points in special position have non-finitely generated pseudoeffective cones of curves.

\begin{cor}\label{cor:infinitely_gen_sigma3}
Assume the SHGH conjecture holds for the blow-up of $\P^2$ at 10 very general points. Let $X=\grass(2+s,5+s)$ and fix a Schubert variety $\Sigma_3 \subset X$. If $\Gamma$ is a set of $r \geq 6$ very general points in $\Sigma_3$, Then $\ceff_1(X_\Gamma)$ is not finitely generated.
\end{cor}
\begin{proof}
We induct on $s$, using Theorem \ref{thm:fanomfld_infinite} as the base case. The Schubert variety $\Sigma_3$ consists of all $(1+s)$-planes in $\P^{4+s}$ passing through a fixed point $p$. Given a curve $C \in \grass(2+s,5+s)$, we obtain a variety
\[
V_C = \bigcup_{W \in C} W,
\]
where the $W$ are treated as $(1+s)$-planes in $\P^{4+s}$. If $(\Sigma_3)_\Gamma$ is the blow-up of the Schubert variety $\Sigma_3$ along $\Gamma$, we obtain a map $\psi$ from $N_1((\Sigma_3)_\Gamma)$ to the blow-up $X^{s-1}_\Delta$ of $\grass(1+s,4+s)$ as follows. Given a curve $C$, we choose a general hyperplane $H \subset \P^{4+s}$. Then $H \cap V_C$ is a subvariety of $\P^{3+s}$ of the same degree as $V_C$. We define the set $\Delta$ to be the intersection of $H$ with the $r$ elements of $\Gamma$.

We also obtain a map $X^{s-1}_\Delta \to X^s_\Gamma$ by sending a $s$-plane $W \subset \P^{3+s}$ to the $(s+1)$-plane in $\P^{4+s}$ spanned by $W$ and the fixed point $p$. This induces a map $\varphi: N_1(X^{s-1}_\Delta) \to N_1(X^s_\Gamma)$.

Let $D \in \ceff_1(X^{s-1}_\Delta)$ be an extremal class. Write $D = a\ell^{s-1} - \sum_{i=1}^r b_i \ell_i$, where $\ell^{s-1}$ is the pull-back of the class $\sigma_{3,\dots,3,2} \in A_1(\grass(1+s,4+s))$. Then $\varphi(D) = a\ell^s  - \sum_{i=1}^r b_i \ell_i$. We claim that $\varphi(D)$ is extremal in $\ceff_1(X_\Gamma^s)$.

First observe that $\varphi(D)$ is pseudoeffective. Indeed, for any $\varepsilon > 0$, $D + \varepsilon \ell^{s-1}$ is effective and is therefore represented by a variety $C_\varepsilon$. This yields a variety $V_{C_\varepsilon} \subset \P^{3+s}$. Taking the cone over $V_{C_\varepsilon}$ we obtain a variety $V_{C_\varepsilon}'$ in $\P^{4+s}$ swept out by the cones over the $(1+s)$-planes of $C_\varepsilon$. Thus $(a+\varepsilon)\ell^s - \sum_i b_i \ell_i$ is effective in $\grass(2+s,5+s)$. Taking the limit as $\varepsilon\to 0$, we see that $\varphi(D)$ is pseudoeffective.

Now suppose $\varphi(D) = \alpha + \beta$ for $\alpha, \beta \in \ceff_1(X^s_\Gamma)$. We next show that $\psi(\alpha)$ and $\psi(\beta)$ are pseudoeffective in $X^{s-1}_\Delta.$ Write $\alpha = c \ell^s - \sum_i d_i \ell_i$. Then $\alpha + \varepsilon\ell^s$ is effective for $\varepsilon > 0$. Let $C_\varepsilon$ be a variety representing $\alpha + \varepsilon \ell^s$. A general hyperplane section of $V_{C_\varepsilon}$ gives a curve in $\grass(1+s,4+s)$ of class $\psi(\alpha + \varepsilon\ell^{s})=(a+\varepsilon)\ell^{s-1} - \sum_i d_i \ell_i$. Taking limits as $\varepsilon \to  0$ shows that $\psi(\alpha)$ is pseudoeffective.

Since $D$ is extremal and $D = \psi(\alpha) + \psi(\beta)$, we must have $\psi(\alpha) = a_1 D$ and $\psi(\beta) = a_2 D$. It follows that $\alpha = a_1\varphi(D)$ and $\beta = a_2 \varphi(D)$.
\end{proof}

The corollary shows that specializing points can easily force pathological behavior of the pseudoeffective cone of curves. Already in the case of $X=\grass(3,6)$, Corollary \ref{prop:lingen_verygeneral} shows that $\ceff_1(X_r)$ is S-generated for $r\leq 42$ very general points. On the other hand, Corollary \ref{cor:infinitely_gen_sigma3} shows that if these points are specialized to a general set $\Gamma\subset\Sigma_3$, then $\ceff_1(X_\Gamma)$ is not finitely generated for $r \geq 6$. Nevertheless, if the points of $\Gamma$ are specialized to particular Schubert varieties, $\ceff_1(X_\Gamma)$ can be well-behaved. For example, if $\Gamma$ is contained in a 1-dimensional Schubert variety then $\ceff_1(X_\Gamma)$ will always be S-generated.

\begin{prop}\label{prop:points_on_a_line}
Let $L=\Sigma_{n-k,n-k,\dots,n-k,n-k-1}$ be a line in $\grass(k,n)$ and $\Gamma$ a set of $r$ points contained in $L$. Then $\ceff_1(X_\Gamma)$ is S-generated.
\end{prop}
\begin{proof}
Let $C \in \ceff_1(X_\Gamma)$. If $r=1$ then the claim follows from Corollary \ref{cor:lin_gen1pt}. If $r \geq 2$, then the base locus of the linear system $H- \sum_{i=1}^r E_i$ is the line $L$. If $C \cdot (H- \sum E_i) < 0$ then $C \subset L$, hence $C = L$. Otherwise, $C$ has non-negative intersection with the linear system, thus is in the span of the Schubert classes.
\end{proof}

It would be interesting to systematically describe effective cones of blow-ups of Grassmannians at special configurations of points. We conclude with a small step in this direction: the S-generation of $\ceff_1(X_\Gamma)$ is determined by S-generation of the effective cone of a blown-up subgrassmannian.

\begin{prop}
Let $V$ be an $n$-dimensional vector space and $W \subseteq V$ a linear subspace. Let $X = \grass(k,V)$ and $Y = \grass(k,W) \subseteq X$. Let $\Gamma$ be a collection of  $r$ points of $Y$, and let $X_\Gamma$ and $Y_\Gamma$ be the blow-ups of $X$ and $Y$ along $\Gamma$, respectively. Then $\ceff_1(X_\Gamma)$ is S-generated if and only if $\ceff_1(Y_\Gamma)$ is.
\end{prop}
\begin{proof}
Let $\ell_X$ denote the pullback to $X_\Gamma$ of the one-dimensional Schubert class on $X$, $\ell_Y$ the same for $Y_\Gamma$, and similar for the exceptional classes $\ell_{i,X}$ and $\ell_{i,Y}$. Assume first that $\ceff_1(X_\Gamma)$ is S-generated. If $C$ is an effective curve class on $Y_\Gamma$, we may write $C = a\ell_Y - \sum_{i=1}^r b_i \ell_{i,Y}$. By Lemma \ref{lemma:positive_coeffs}, we may assume $b_i \geq 0$ for all $1 \leq i \leq r$. If $i:Y_\Gamma \to X_\Gamma$ is the inclusion then $i_\ast \ell_Y = \ell_X$ and $i_\ast \ell_{i,Y} = \ell_{i,X}$. Since $\ceff_1(X_\Gamma)$ is S-generated and $i_\ast C$ is effective, we have $a \geq \sum_{i=1}^r b_i$. It follows that $C$ is itself in the span of Schubert classes on $Y_\Gamma$.

Conversely, assume $\ceff_1(Y_\Gamma)$ is S-generated and let $C$ be a curve in $X_\Gamma$. Then either $C$ has non-negative intersection with $H-\sum_{i=1}^r E_i$, in which case $C$ is in the span of the Schubert classes, or $C$ is contained in the base locus. Note that the base locus of $H-\sum_{i=1}^r E_i$ is contained in $Y$. Indeed, in the Pl\"ucker embedding $X \to \P(\Lambda^k V)$, we have $Y = X \cap \P(\Lambda^k W)$. In particular, the base locus of the space of hyperplanes in $\P(\Lambda^k V)$ containing $\Gamma$ is contained in $\P(\Lambda^k W)$, hence its intersection with $X$ is contained in $Y$. It follows that if $C$ has negative intersection with $H- \sum_{i=1}^r E_i$ then $C \subset Y$, in which case $C$ is in the span of the Schubert classes by assumption.
\end{proof}

\begin{example}
Let $X = \grass(2,V)$ with $n=\dim V>4$. Let $\Gamma$ be a set of $r$ general points of a Schubert variety $\Sigma_{n-3,n-4}(F_\bullet)$ for some flag $F_\bullet$. We will show that $X_\Gamma$ is S-generated if and only if $r \leq 2$. By definition, $\Sigma_{n-3,n-4}(F_\bullet)$ consists of 2-planes contained in $F_4$ whose intersection with $F_2$ is at least 1-dimensional. Then $\Sigma_{n-3,n-4}(F_\bullet)$ is contained in the Grassmannian $Y = \grass(2,F_4)$. By the proposition, it suffices to show $\ceff_1(Y_\Gamma)$ is S-generated if and only if $r \leq 2$. The claim now follows from Proposition \ref{prop:curves_g24}.
\end{example}

\bibliographystyle{plain}

\begin{thebibliography}{ABCD}
\bibitem[A]{anderson}
D.~Anderson, Effective divisors on Bott-Samelson varieties. arXiv:1501.00034.

\bibitem[CC]{chen-coskun}
D.~Chen and I.~Coskun, Extremal higher codimension cycles on moduli spaces of curves. \emph{Proc. Lond. Math. Soc.} \textbf{111} no.\ 1 (2015), 181--204.

\bibitem[C1]{coskun-lwr}
I.~Coskun, A Littlewood-Richardson rule for two-step flag varieties. \emph{Invent. Math.}, \textbf{176} no.\ 2 (2009), 325--395.

\bibitem[C2]{coskun-rigidity}
I.~Coskun, Rigid and non-smoothable Schubert classes. \emph{Journal of Differential Geometry}, \textbf{87} (2011), no.\ 3, 493--514.

\bibitem[CCH+]{degree-grassmannians} I.~Coskun, L.~Costa, J.~Huizenga, R.M.~Mir\'o-Roig, and M.~Woolf, Ulrich Schur bundles on flag varieties. arXiv:1512.06193.

\bibitem[CLO]{clo}
I.~Coskun, J.~Lesieutre, and J.C.~Ottem, Effective cones of cycles of blow-ups of projective space. \emph{J. of Algebra \& Number Theory}, \textbf{10} (2016), 1983--2014.

\bibitem[CPS]{cps}
I.~Coskun and A.~Prendergast-Smith, Fano manifolds of index $n-1$ and the cone conjecture. \emph{Int. Math. Res. Notices}, \textbf{9} (2014), 2401--2439.

\bibitem[DF]{de-fernex} T.~de~Fernex, On the Mori cone of blow-ups of the plane. arXiv:1001.5243.

\bibitem[DELV]{delv}
O.~Debarre, L.~Ein, R.~Lazarsfeld, and C.~Voisin, Pseudoeffective and nef classes on abelian varieties. \emph{Compos. Math.} \textbf{147} (2011), no.\ 6, 1793--1818.

\bibitem[DJV]{djv}
O.~Debarre, Z. Jiang, and C. Voisin, Pseudo-effective classes and pushforwards. \emph{Pure Appl. Math. Q.} \textbf{9} (2013), no.\ 4, 643--664.

\bibitem[EH]{eisenbud-harris}
D.~Eisenbud and J.~Harris, \emph{3264 and All That: A Second Course in Algebraic Geometry}. Cambridge University Press, Cambridge, 2016.

\bibitem[F1]{fujita-structure}
T.~Fujita, On the structure of polarize manifolds with total deficiency one, I. \emph{J. Math. Soc. Japan}, \textbf{32}, no.\ 4 (1980), 709--725.

\bibitem[F2]{fujita-classification}
T.~Fujita, On polarized varieties of small $\Delta$-genera. \emph{Tohoku Math. J. (2)}, \textbf{34}, no.\ 3 (1982), 319--341.

\bibitem[F3]{fujita-classification2}
T.~Fujita, Classification of projective varieties of $\Delta$-genus one. \emph{Proc. Japan. Acadm. Ser. A. Math. Sci.}, \textbf{58} (1982), 113--116.

\bibitem[FL]{fulger-lehmann}
M.~Fulger and B.~Lehmann, Positive cones of dual cycle classes. \emph{J. Alg. Geom.}, to appear. arXiv:1408.5154.

\bibitem[Ful]{fulton}
W.~Fulton, \emph{Young Tableaux: With Applications to Representation Theory and Geometry}. London Mathematical Society Student Texts, 35, Cambridge University Press, Cambridge, 1998.

\bibitem[FMSS]{fmss}
W.~Fulton, R.~MacPherson, F.~Sottile, and B.~Sturmfels, Intersection theory on spherical varieties. \emph{J. Alg. Geom.}, \textbf{4} (1994), 181--193.

\bibitem[G]{gimi}
A.~Gimigliano, On Linear Systems of Plane Curves, Ph.D. Thesis, Queen's University, 1987.

\bibitem[Har]{harbourne}
B.~Harbourne, The geometry of rational surfaces and Hilbert functions of points in the plane. \emph{Can. Math. Soc. Conf. Proc.} \textbf{6} (1986), 95--111.


\bibitem[Hi]{hirschowitz}
A.~Hirschowitz, Une conjecture pour la cohomologie des diviseurs sur les surfaces rationelles g\'en\'eriques. \emph{J. Reine Angew. Math.} \textbf{397} (1989), 208--213.

\bibitem[I1]{iskovskih-3foldsi}
V.A.~Iskovskih, Fano 3-folds I. \emph{Math. USSR-Izv.}, \textbf{11} (1977), 485--527.

\bibitem[I1]{iskovskih-3foldsii}
V.A.~Iskovskih, Fano 3-folds II. \emph{Math. USSR-Izv.}, \textbf{12} (1978), 469--506.


\bibitem[L]{lazarsfeld} R. Lazarsfeld, \emph{Positivity in Algebraic Geometry I, Classical Setting: Line Bundles and Linear Series}, Springer-Verlag, 2004.

\bibitem[R]{rischter} R. Rischter, Log Fano varieties and Mori Dream Spaces, Ph.D. Thesis, IMPA, preprint.

\bibitem[RZ]{rosenthal-zelevinsky}
J.~Rosenthal and A.~Zelevinsky, Multiplicities of points on Schubert varieties in Grassmannians. \emph{J. Algebraic Combinatorics}, \textbf{13} (2001), 213--218.

\end{thebibliography}

\end{document}